\documentclass{article}
\usepackage[numbers,sort&compress]{natbib}
\usepackage{amsmath}
\usepackage{amssymb}
\usepackage{multirow}
\usepackage{amsfonts}
\usepackage{mathrsfs}
\usepackage{booktabs}
\usepackage{authblk}
\usepackage{algorithmicx}
\usepackage{algorithm}
\usepackage{booktabs}
\usepackage{algpseudocode}
\usepackage{comment}
\usepackage{amsthm}
\usepackage{latexsym}
\usepackage{graphicx}
\usepackage{rotating}
\usepackage{subfig}
\usepackage{color}
\usepackage{varioref}
\usepackage{longtable}
\usepackage{epstopdf}
\allowdisplaybreaks[4]
\usepackage[colorlinks,linkcolor=blue]{hyperref}

\newtheorem{theorem}{Theorem}
\newtheorem{assumption}{Assumption}
\newtheorem{proposition}{Proposition}
\newtheorem{definition}{Definition}
\newtheorem{lemma}{Lemma}
\newtheorem{remark}{Remark}

\usepackage{geometry}
\newcommand{\blue}[1]{\begin{color}{blue}#1\end{color}}

\begin{document}

\title{Multi-step Inertial Accelerated  Doubly Stochastic Gradient Methods for 
Block Term Tensor Decomposition}


\author{Zehui Liu\thanks{LMIB, School of Mathematical Sciences, Beihang University, Beijing 100191, China. Email: liuzehui@buaa.edu.cn}, \hskip 0.2cm 
Qingsong Wang\thanks{School of Mathematics and Computational Science, Xiangtan University, Xiangtan 411105, China. Email: nothing2wang@hotmail.com}, \hskip 0.2cm 
Chunfeng Cui\thanks{LMIB, School of Mathematical Sciences, Beihang University, Beijing 100191, China. Email: chunfengcui@buaa.edu.cn}
}

\maketitle

\begin{abstract} 
In this paper, we explore a specific optimization problem that combines a differentiable nonconvex function with a nondifferentiable function for multi-block variables, which is particularly relevant to tackle the multilinear rank-($L_r$,$L_r$,1) block-term tensor decomposition model with a regularization term. While existing algorithms often suffer from high per-iteration complexity and slow convergence, this paper employs a unified multi-step inertial accelerated doubly stochastic gradient descent method tailored for structured rank-$\left(L_r, L_r, 1\right)$ tensor decomposition, referred to as  Midas-LL1. We also introduce an extended multi-step variance-reduced stochastic estimator framework. Our analysis under this new framework demonstrates the subsequential and sequential convergence of the proposed algorithm under certain conditions and illustrates the sublinear convergence rate of the subsequence, showing that the Midas-LL1 algorithm requires at most $\mathcal{O}(\varepsilon^{-2})$ iterations in expectation to reach an $\varepsilon$-stationary point. The proposed algorithm is evaluated on several datasets, and the results indicate that Midas-LL1 outperforms existing state-of-the-art algorithms in terms of both computational speed and solution quality.

\end{abstract}


\maketitle
\section{Introduction}

In this paper, we focus on solving the composite optimization problem for multi-block variables, formulated as follows:

\begin{equation}\label{Pure_problem}
\min _{\{x_n\}_{n=1}^{N}} \Phi \left({x}_1, \ldots, {x}_N\right) \equiv f\left({x}_1, \ldots, {x}_N\right)+\sum_{n=1}^N h_n\left({x}_n\right),
\end{equation}
where $x_n \in \mathbb{R}^{d_n}, f: \Pi_{n=1}^N \mathbb{R}^{d_n} \rightarrow \mathbb{R}$ is a continuously differentiable nonconvex function for the entire variable set $x=\left(x_1, \ldots, x_N\right)$, and convex concerning each block $x_n$ when the other blocks are fixed, i.e.,  {$f$} is a block multiconvex function. The function $f(x)$ has a finite-sum structure, expressed as $f(x)=\frac{1}{n} \sum_{i=1}^n f_i(x)$, and $h_n$ for $n=1, \ldots, N$, are extended-value convex functions that serve as regularization terms, capturing prior information about $x_n$, such as non-negativity \cite{Lim2009}. We consider a class of problems that are nonsmooth and nonconvex. Optimization problems of the form \eqref{Pure_problem} are particularly relevant in practical applications and are widely used across various fields of science and engineering, including blind source separation\cite{JUTTEN19911}, nonnegative matrix and tensor factorization \cite{ComonGLM08, KoldaB09, CheW20, WangLCH24}, and Poisson inverse problems \cite{BolteSTV18First}.

Several methods have been developed to tackle these types of optimization problems. One state-of-the-art approach is the block coordinate descent (BCD) \cite{Tseng2001} method of Gauss-Seidel type, which iteratively optimizes each block of variables while keeping the others fixed, effectively breaking down a complex problem into simpler subproblems. Additionally, BCD allows these blocks to be updated via proximal minimization, as discussed in \cite{auslender1992}. However, it is important to note that both BCD and proximal BCD schemes generally involve an inner minimization problem. This can lead to a higher computational burden, particularly when the minimizer does not have a closed-form solution. 
In 2013, Xu and Yin \cite{120887795} proposed an efficient {proximal linearized BCD} approach, which leverages the linearization of the differentiable component function $f$, facilitating the development of efficient numerical algorithms.  Bolte, Sabach, and Teboulle \cite{BolteST14} proposed the proximal alternating linearized minimization (PALM) for blocks $N\ge2$, where the blocks are taken
in cyclic order while fixing the previously computed iterate. The proximal linearized BCD/PALM is briefly outlined below:
\begin{equation}\label{palm}
    \mathbf{x}_n^{k+1}=\underset{\mathbf{x}_n \in {\mathbb{R}^{d_n}}}{\operatorname{argmin}}\left\langle \nabla f_n^{k+1}\left(\mathbf{x}_n\right), \mathbf{x}_n-\mathbf{x}_n^{k}\right\rangle+\frac{1}{2\eta^k}\left\|\mathbf{x}_n-\mathbf{x}_n^{k}\right\|^2+h_n\left(\mathbf{x}_n\right),
\end{equation}
where $f_n^{k+1}\left(\mathbf{x}_n\right) \triangleq f\left(\mathbf{x}_1^{k+1}, \ldots, \mathbf{x}_{n-1}^{k+1}, \mathbf{x}_n, \mathbf{x}_{n+1}^{k}, \ldots, \mathbf{x}_N^{k}\right)$ for {all} $n$, and the stepsize $\eta^k$ is computed according to  Lipschitz constants of $\nabla f_n^{k+1}\left(\mathbf{x}_n\right)$. Based on the Kurdyka-{\L}ojasiewicz inequality, the bounded sequence $\left\{\mathbf{x}_1^{k},\ldots,\mathbf{x}_N^{k}\right\}_{k=0}^{\infty}$ generated by \eqref{palm} can converge globally to the critical point of \eqref{Pure_problem}. 

Stochastic algorithms can reduce iteration costs when the data dimension is large. Xu and Yin \cite{XuY15}  integrated  a  {vanilla} stochastic gradient descent (SGD) estimator with PALM, and Driggs et al. \cite{DriggsTLDS2020} introduced the SPRING algorithm, a stochastic version of PALM, using more sophisticated   variance-reduced gradient estimators
instead of the  {vanilla} stochastic gradient descent estimators. Additionally,  Hertrich and Steidl \cite{HertrichS20} proposed a stochastic version of iPALM \cite{PockS16} that introduced an inertial step to improve the performance of PALM further, and established its convergence under similar conditions. Liang et al. \cite{Liang2016} introduced a multi-step inertial forward-backward splitting algorithm aimed at minimizing the sum of two non-necessarily convex functions. Furthermore, Guo et al. \cite{Guo2023} developed a stochastic two-step inertial Bregman PALM algorithm for large-scale nonconvex and nonsmooth optimization problems.

These advancements in stochastic algorithms are particularly relevant for high-dimensional data problems. Specifically, the form of problem \eqref{Pure_problem} can be effectively applied to tensor decomposition \cite{KoldaB09} with regularization. Tensor decompositions generalize matrix decompositions and are powerful tools that enable the recovery of underlying components. BCD and tensor decomposition are linked together through the matrix unfolding operation, rearranging the elements of a tensor to a matrix and pushing the latent factors to the rightmost position in the unfolded tensor representation. In 2020, Fu et al. \cite{FuIWGH20} proposed BrasCPD, a block-randomized stochastic algorithmic framework for computing the CP decomposition \cite{Hitchcock1927} of large-scale dense tensors. Wang et al. introduced mBrasCPD \cite{WangCH21} and iBrasCPD \cite{Wang2023}, which accelerate the SGD scheme using the heavy ball method \cite{Polyak64} and inertial acceleration, respectively. Additionally, Pu et al. \cite{Pu2022} developed a block-randomized stochastic mirror descent (SMD) algorithmic framework for large-scale generalized CP decomposition (GCP) \cite{GCP2020}. 
Recently, Liu et al. \cite{liu2024} developed an inertial accelerated version of SMD for GCP problem, and showed sub-sequential and sequentially convergence guarantees.
Besides CP decomposition, many tensor factorization models, such as Tucker decomposition \cite{Tucker1963, Tucker1964, Tucker1966} and block term decomposition (BTD) \cite{DeLathauwer2008a, DeLathauwer2008b, DeLathauwer2008c}, exhibit similar multi-linearity properties in their respective unfolded forms. However, the aforementioned algorithms primarily focus on solving CP decomposition problems, with limited attention given to other types of tensor decomposition problems.

CP decomposition and Tucker decomposition correspond to different tensor generalizations of matrix rank. Tucker decomposition is associated with $n$-ranks, generalizing column and row ranks. In contrast, CP decomposition defines rank as the minimal number of rank-1 terms required to represent a tensor, treating all tensor modes equally and processing them identically. However, Tucker decomposition does not explicitly divide a tensor into a sum of component tensors, complicating its direct application to the linear mixture model. In 2008, the Tucker and CP decompositions were connected by introducing the BTD. This model overcomes the CP decomposition limitation that each component must be rank-1 and the Tucker decomposition constraint of having a single component tensor. Consequently, BTD has found extensive application in hyperspectral unmixing \cite{bioucas2012, ma2014signal}, and community detection in networks \cite{gujral2020}. In \cite{DeLathauwer2008b}, BTD was introduced as a sum of $R$ rank-$\left(L_r, M_r, N_r\right)$ terms $(r=1,2, \ldots, R)$ in general. Domanov et al. \cite{23M1557246} recently presented new deterministic and generic conditions for the uniqueness of rank-$\left(L_r, M_r, N_r\right)$ decompositions. Another special case of rank-$\left(L_r, L_r, 1\right)$ BTD for third-order tensor has attracted more attention, because of both its more frequent occurrence in applications and the existence of more concrete and easier to check uniqueness conditions.  For instance, a hyperspectral image (HSI) can be represented as a third-order tensor, where its modes are distinguishable in a spectral-spatial manner, with two spatial modes separate from a spectral mode. Therefore, rank-$\left(L_r, L_r, 1\right)$ BTD is a suitable model, where the first two modes correspond to the spatial positions, and the third mode corresponds to the spectral information. This paper shall also focus on this specific and widely-used rank-$\left(L_r, L_r, 1\right)$ BTD model. 

The rank-$\left(L_r, L_r, 1\right)$ BTD often have structural constraints and regularization on the latent factors, which usually come from physical meaning and prior information on HSIs. Designing tailored algorithms for structured rank-$\left(L_r, L_r, 1\right)$ tensor decomposition presents several challenges. One key issue is that the speed of existing methods is often inadequate. For example, Alternating Least Squares (ALS) has been extended to compute tensor BTD as described in \cite{DeLathauwer2008c}. However, the primary bottleneck in ALS lies in the matricized-tensor times Khatri-Rao product (MTTKRP) \cite{fu2019}, which becomes even more computationally demanding in the BTD context compared to its application in CP decomposition. This increased computational cost arises because two of the latent factors in BTD are typically larger in scale. Nonnegativity constraints on the latent factors are typically managed by the classic multiplicative update (MU) algorithm \cite{lee2001}, which updates one factor using the majorization-minimization method but with a conservative stepsize. The MU algorithm is also prone to numerical issues in some cases, when there are iterates that contain zero elements \cite{4359171}. In addition, the ALS-MU \cite{qian2017, 8497054, 6797154} combination often leads to a considerably high per-iteration complexity. These algorithms, however, also encounter challenges related to slow convergence issues.
These limitations remain a critical area for ongoing research, prompting us to explore the potential of stochastic methods in computing rank-$\left(L_r, L_r, 1\right)$ BTD. 

 In this paper, a doubly stochastic gradient descent method with multi-step inertial acceleration is proposed to tackle the nonconvex and nonsmooth optimization problem \eqref{Pure_problem}. This method is particularly suitable for applications in rank-$\left(L_r, L_r, 1\right)$ BTD. Our detailed contributions are as follows:

\begin{itemize}
\item[(i)] We introduce a multi-step inertial accelerated block-randomized stochastic gradient descent method, i.e. doubly stochastic algorithm for rank-$\left(L_r, L_r, 1\right)$ block-term decomposition problem, denoted by Midas-LL1. Inherited from the variance-reduced stochastic estimator, we propose an extended multi-step and multi-block variance-reduced stochastic estimator. Based on these, we demonstrate that Midas-LL1 achieves a sublinear convergence rate for the generated subsequence.

\item[(ii)] The global convergence of the sequence generated by Midas-LL1 is established. We introduce a novel muti-step Lyapunov function, showing that the algorithm requires at most $\mathcal{O}(\varepsilon^{-2})$ iterations in expectation to attain an $\varepsilon$-stationary point. 

\item[(iii)] We conduct extensive experiments on two hyperspectral image datasets and two video datasets to illustrate the effectiveness of our proposed Midas-LL1 algorithms. Our numerical results indicate that incorporating variance-reduced stochastic gradient estimators and multi-step acceleration into Midas-LL1 yields superior performance.
\end{itemize}

The rest of this paper is organized as follows. Section \ref{preliminary} provides the essential definitions and preliminary results related to existing models and algorithms. Section \ref{algorithm} introduces the formulation of the proposed Midas-LL1 algorithm. Section \ref{convergence_analysis} focuses on establishing the convergence properties and convergence rate of Midas-LL1. In Section \ref{numercial_experiments}, we evaluate the performance of the Midas-LL1 algorithm in comparison with several baseline methods using real-world datasets. Finally, the paper concludes in Section \ref{conclusion}.

\section{Preliminaries}\label{preliminary}
 
In this section, we summarize some useful definitions and introduce the block term decomposition as well as various stochastic methods for tensor decomposition.

\begin{definition}\cite{WetsR98}
Let $h:  \mathbb{R}^{d} \rightarrow(-\infty,+\infty]$ be a proper and lower semicontinuous function.  
For $x \in \operatorname{dom} h$, the Fréchet subdifferential of $\Phi$ at $x$, written $\hat{\partial} h(x)$, is the set of vectors $v \in \mathbb{R}^d$ which satisfy
$$
\liminf _{y \rightarrow x} \frac{1}{\|x-y\|_2}[h(y)-h(x)-\langle v, y-x\rangle] \geq 0 .
$$
If $x \notin \operatorname{dom} h$, then $\hat{\partial} h(x)=\emptyset$. The limiting subdifferential of $h$ at $x \in \operatorname{dom} h$, written $\partial h(x)$, is defined as follows:
$$
\partial h(x):=\left\{v \in \mathbb{R}^d: \exists x_k \rightarrow x, h\left(x_k\right) \rightarrow h(x), v_k \in \hat{\partial} h\left(x_k\right), v_k \rightarrow v\right\} .
$$
\end{definition}


\begin{proposition}
 For all $(x_{1},\dots,x_{s}) \in \operatorname{dom} \Phi$ in problem \eqref{Pure_problem} we have
 $$
\partial \Phi(\mathbf{x})=\left\{\nabla_{\mathbf{x}_1} f(\mathbf{x})+\partial h_1\left(\mathbf{x}_1\right)\right\} \times \cdots \times\left\{\nabla_{\mathbf{x}_s} f(\mathbf{x})+\partial h_s\left(\mathbf{x}_s\right)\right\}.
$$
\end{proposition}
\begin{proof}
    Observe first that we can get  $\partial \Phi(\mathbf{x})= \nabla f(\mathbf{x})+\partial\left(\sum_{n=1}^s h_n\right)$,  since $f$ is continuously differentiable. Further, the subdifferential calculus for separable functions  yields $\partial\left(\sum_{n=1}^s h_n\right)=\partial h_1\left(\mathbf{x}_1\right) \times \cdots \times  \partial h_s\left(\mathbf{x}_s\right)$. Hence we get the above equality.
\end{proof}

\begin{definition} \label{stationary-point} (\cite{Lan2020First} $\epsilon$-stationary point) 
Given $\epsilon>0$, a solution $\{x_{1}^{*},\dots,x_{s}^{*}\}$ is said to be an $\epsilon$-stationary point of function $\Phi(x_{1},\dots,x_{s})$ if
\[
\mbox{dist}(0,\partial \Phi(x_{1}^{*},\dots,x_{s}^{*}))\le\epsilon. 
\]
\end{definition} 

\subsection{Block term decomposition with rank-($L_r$,$L_r$,1)}
Tensor decomposition \cite{KoldaB09} can break a large-size tensor into many small-size factors, including CP \cite{Hitchcock1927, Hitchcock1928} and Tucker decomposition \cite{Tucker1963, Tucker1964, Tucker1966} connected with two different tensor generalizations of matrix decomposition. BTD \cite{DeLathauwer2008a, DeLathauwer2008b, DeLathauwer2008c} is proposed to unify the Tucker and CP, decomposing a tensor into a sum of low multilinear rank terms.  

Consider an $N$-th order tensor $\mathcal{X}\in\mathbb{R}^{I_{1}\times I_{2}\times\dots\times I_{N}}$ with  $\mathcal{X}(i_{1},\dots,i_{N})$ represents an element of $\mathcal{X}$. 
The block term decomposition is
\begin{equation}\label{BTD}
  \mathcal{X}=\sum_{r=1}^R \mathcal{G}_r \times_1 \mathbf{U}_r^{(1)} \times_2 \mathbf{U}_r^{(2)} \cdots \times_N \mathbf{U}_r^{(N)},  
\end{equation}
in which $\mathbf{U}_r^{(n)} \in \mathbb{R}^{I_n \times J_n^r}$ represents the $n$-th factor in the $r$-th term, and $\mathcal{G}_r \in$ $\mathbb{R}_1^{J_1^r \times J_2^r \times \cdots \times J_N^r}$ is the core tensor in the $r$-th term.  If $\mathcal{G}_r$ is an identity cubical tensor for all $r=1,\dots,R$, then \eqref{BTD} becomes the CP decomposition. When \(R=1\),  \eqref{BTD} reduces to Tucker decomposition. A visual representation of BTD  for a third-order tensor is shown in Figure \ref{BTD_figure}. 
\begin{figure}[!htb]
\setlength\tabcolsep{2pt}
\centering
\begin{tabular}{cccccc}
\includegraphics[width=0.8\textwidth]{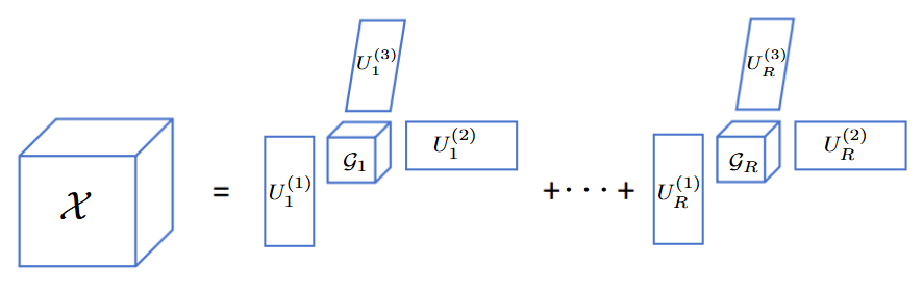}
\end{tabular}
\caption{BTD for a third-order tensor.}
\label{BTD_figure}
\end{figure}

The mode-$n$ {\sl unfolding} of $\mathcal X$, denoted by $X_{(n)}$, is a $J_n$-by-$I_n$ matrix, where $J_{n}=\prod_{m=1,m\neq n}^{N}I_{m}$ and each element in $X_{(n)}$ is defined by 
\[
X_{(n)}(j,i_{n})=\mathcal{X}(i_{1},\dots,i_{N}),
\]
where $j=1+\sum_{k=1,k\neq n}^{N}(i_{k}-1)\bar{J}_{k}$ and $\bar{J}_{k}=\prod_{m=1,m\neq n}^{k-1}I_{m}$ \cite{KoldaB09}. 
A mode-$n$ fiber is defined by fixing every index $(i_1,\dots,i_{n-1},i_{n+1},\dots,i_N)$ but keeping $i_n\in\{1,\dots,I_n\}$ in variation.

Under the umbrella of block term decomposition \eqref{BTD}, the special case of BTD with rank-($L_r$, $L_r$, 1) terms has earned significant interest due to its frequently occurring application and the easily verifiable conditions for uniqueness. 
 The BTD with rank-$\left(L_r, L_r, 1\right)$  
 approximates a third-order tensor as a sum of $R$ component tensors, where each term is the outer product of a rank-$L_r$ matrix and a vector. Specifically, let $A_1^r \in \mathbb{R}^{I_1 \times L_r}$ and $A^r_2 \in \mathbb{R}^{I_2 \times L_r}$ be rank-$L_r$ matrices, and $\boldsymbol{c}_r \in \mathbb{R}^{I_3}$ be a non-zero vector. Then, the rank-$(L_r, L_r, 1)$ BTD of $\mathcal{X}\in \mathbb{R}^{I_1 \times I_2 \times I_3}$ can be expressed as follows:
 
\begin{equation}\label{LL1}
  \mathcal{X} \approx \sum_{r=1}^R(A_{1,r} \cdot A_{2,r}^{\top}) \circ \boldsymbol{c}_r. 
\end{equation}
\begin{figure}[!htb]
\setlength\tabcolsep{2pt}
\centering
\begin{tabular}{cccccc}
\includegraphics[width=0.8\textwidth]{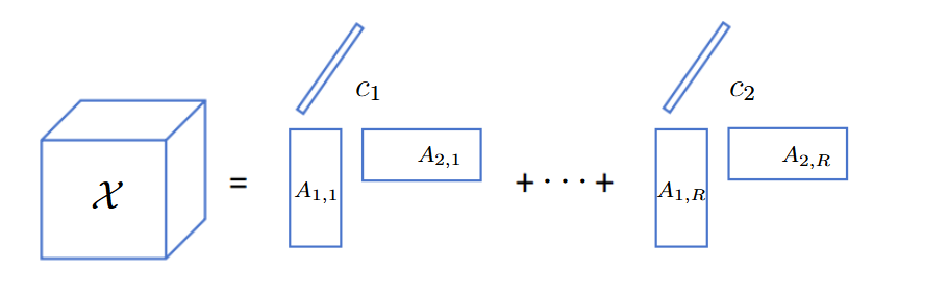}
\end{tabular}
\caption{BTD with rank-$\left(L_r, L_r, 1\right)$ for a third-order tensor.}
\label{LL1_figure}
\end{figure}
A visual representation of this decomposition for a third-order tensor is shown in Figure \ref{LL1_figure}. The Canonical Polyadic Decomposition (CPD) becomes a specific case of BTD with rank-$(1, 1, 1)$ terms when $L_r=1$. Consequently, permutation and scaling indeterminacies inherited from the CPD, a nice feature of this decomposition is that the rank-one block-terms $(A^r_1 \cdot {A_2^r}^{\top}) \circ \boldsymbol{c}_r$ are unique under quite mild conditions. Therefore, the rank-$\left(L_r, L_r, 1\right)$ BTD is also called essentially unique when
it is subject only to these trivial indeterminacies.
We use the following interpretation to
formulate \eqref{LL1} as the nonlinear least squares problem
\begin{equation}\label{def:f}
	f(A_1, A_2, A_3) = \frac{1}{2I^N}\left\|\mathcal{X}-\sum_{r=1}^R(A_{1,r} \cdot {A_{2,r}}^{\top}) \circ \boldsymbol{c}_r\right \|_{F}^{2}, 
\end{equation}
where the partitioned matrices $A_1=[A_{1,1},\dots,A_{1,R}]\in \mathbb{R}^{I_1 \times \sum_{r=1}^R L_r}$, $A_2=[A^1_2,\dots,A^R_2]\in \mathbb{R}^{I_2 \times \sum_{r=1}^R L_r}$, and $A_3=[\boldsymbol{c}_1,\dots,\boldsymbol{c}_R]\in \mathbb{R}^{I_3 \times R}$. Denote $I=(I_1I_2I_3)^{\frac13}$ as the geometric mean of the three dimensions. 
The function $f(\cdot)$ is smooth, yet it is both nonlinear and non-convex. The corresponding loss function is derived from a maximum likelihood estimate, predicated on the assumption that the data adhere to a Gaussian distribution. 
Alternative assumptions regarding the data distributions, such as those based on Bernoulli or Poisson distributions, would necessitate the formulation of distinct loss functions. For additional details, we refer the reader to \cite{GCP2020}. 

By taking the structures of the variables $A_n$ into consideration, we focus on the regularized problem as follows
\begin{align}\label{regular_LL1D}
	\min_{\{A_{n}\}_{n=1}^{3}}\,\, \Phi(A_{1},A_{2},A_{3}) := f(A_{1}, A_2,A_{3})+\sum_{n=1}^{3}h_{n}(A_{n}).
\end{align}
Here, $h_{n}(A_{n})$ denotes a structure promoting regularizer on $A_{n}$, such as column-wise orthogonality \cite{Krijnen2008}, Tikhonov regularization \cite{Paatero2000},  nonnegativity \cite{Lim2009}. Those regularizations may result in well-posed problems. For instance, if $A_{n}\in \mathbb{R}^{I_n\times R}_{+}:=\{A_{n}\vert A_{n}\ge 0\}$ is applied, we can write $h_{n}(\cdot)$ as the indicator function: 
\begin{equation*}\label{nonnegative_eq}
	h_{n}(A)=\mathcal{I}_{\mathbb R_+^{I_n\times R}}\left(A\right)=\left\{\begin{array}{ll}
		0, & A \ge0, \\
		\infty, & \text { otherwise. }
	\end{array}\right.
\end{equation*}
Suppose all factors but $A_n$ are fixed, then  (\ref{regular_LL1D}) is reduced to the subproblem
\begin{eqnarray}\label{matrix_ALS}
 \min_{A_n}\quad\frac{1}{2I^N}\left\|X_{(n)}-H_{n}A_n^{\top}\right\|_{F}^{2}+h_n(A_n), 
\end{eqnarray} 
where $H_{n}$ for $n=1,2$, and $3$ is defined by
\begin{eqnarray}\label{H_12}
    \begin{aligned}
    H_{1} &=A_{3} \odot_b A_{2}=[c_1 \otimes A^1_2, \ldots, c_R \otimes A^R_2]\in\mathbb{R}^{J_{1}\times L}, \\ 
    H_{2} &=A_{3} \odot_b A_{1}=[c_1 \otimes A^1_1, \ldots, c_R \otimes A^R_1]\in\mathbb{R}^{J_{2}\times L},\\
    H_{3} &=\left[\left(A^1_2 \odot A^1_1\right) 1_{L_1}, \ldots,\left(A^R_2 \odot A^R_1\right) 1_{L_R}\right]\in 
\mathbb{R}^{J_{3}\times R},
\end{aligned}
\end{eqnarray}
where $L=\sum_{r=1}^R L_r$  and $1_{L_r}$ is all-one column vector with length $L_r$. The notation $\odot_b$ is a generalized block-wise Khatri–Rao product for partitioned matrices with the same number of submatrices, which consists of the partition-wise Kronecker product. The formulation of $H_{3}$ is  due to the last factor $A_3$ of BTD is the rank-1 component.

Optimization methods such as (proximal) gradient descent can solve the subproblem \eqref{matrix_ALS} directly. The gradient of $f$ with respect to $A_n$ is equal to 
\begin{equation}\label{equ:grad_f}
	\nabla_{A_n} f(A_{1},A_{2},A_{3})=\frac{1}{I^N}(A_{n}  H_{n}^{\top}H_{n}-X_{(n)}^{\top}H_{n}). 
\end{equation}
Computing $X_{(n)}^{\top}H_{n}$ in \eqref{equ:grad_f} is usually expensive. It is referred to the matricized-tensor times  Khatri-Rao product (MTTKRP) \cite{KoldaB09}, which takes $\mathcal O(RI_1 I_{2} I_3)$ operations. When the value of $I_{n}$ ($n=1,\dots,3$) is large, the cost of computing the gradient \eqref{equ:grad_f} is often prohibitively expensive, rendering most traditional deterministic first-order optimization algorithms ineffective.
Rank-$\left(L_r, L_r, 1\right)$ BTD presents a distinct structural configuration, meriting further exploration in applying the stochastic gradient method.

\subsection{Stochastic methods with inertial acceleration for tensor decomposition}

If $L_r = 1$, the decomposition format in \eqref{LL1} is CP decomposition and \eqref{matrix_ALS} is a CP decomposition type subproblem.  
    The iBrasCPD \cite{Wang2023} updates the factor variables by 
\begin{eqnarray}
	\begin{aligned}
		A_{n}^{k+1} &= \underset{A_{n}}{\arg\min}\,\,h_{n}\left(A_{n}\right)+\langle \tilde{\nabla}_{A_n}f(\underline{A}_{n}^{k}),A_{n}-\tilde A_{n}^{k}\rangle+\frac{1}{2\eta^{k}} \left\|A_{n}-\tilde A_{n}^{k}\right\|_{F}^{2}, \\
		A_{n^{\prime}}^{k+1} &= A_{n^{\prime}}^{k}, \quad n^{\prime} \neq n,
	\end{aligned}\label{update_A_O}
\end{eqnarray}
where $\tilde{\nabla}_{A_n}f(\underline{A}_{n}^{k})=\tilde{\nabla}_{A_n}f(A_1^k \cdots \underline{A}_n^k \cdots A_N^k)$ with inertial acceleration $\underline{A}_{n}^{k}=A_{n}^{k}+\beta^{k}(A_{n}^{k}-A_{n}^{k-1})$ and $\tilde{A}_{n}^{k}={A}_{n}^{k}+\alpha^{k}(A_{n}^{k}-A_{n}^{k-1})$ for proximal term. When  $\alpha^{k}=\beta^{k}=0$, then the iBrasCPD algorithm is reduced to BrasCPD \cite{FuIWGH20}; if $\alpha^{k}=0$, then the iBrasCPD equals to mBrasCPD \cite{WangCH21}. 

If $h_{n}(\cdot)$ is a closed proper convex function,    \eqref{update_A_O} can be solved by applying the proximal operator of $h_{n}(\cdot)$, which is  denoted as 
\begin{eqnarray}
	A_{n}^{k+1}= \mbox{Prox}_{\eta^{k}h_{n}}\left(\tilde A_{n}^{k}-\eta^{k} \tilde{\nabla}_{A_n}f(\underline{A}_{n}^{k})\right). \label{h_proximal}
\end{eqnarray}

The variance-reduced gradient estimators are also taken into consideration. One of the popular variance-reduced gradient estimators is SAGA \cite{DefazioBL14},  which is formulated as follows 
\begin{eqnarray}
	\tilde{G}_{n}^{k}:=\frac{1}{{I_n}\left|\mathcal{F}_{n}^{k}\right|}\left(\sum_{j\in \mathcal{F}_{n}^{k}}\nabla_{A_{n}}f_{j}(A^{k})-\nabla_{A_{n}}f_{j}((\tilde{A}^{k})^{j})\right)+\frac{1}{J_{n}}\sum_{i=1}^{J_{n}}\nabla_{A_{n}}f_{i}((\tilde{A}^{k})^{i}),\label{vr_gradient}
\end{eqnarray}
where $A^{k}:=(A_{1}^{k},\dots,A_{N}^{k})$, and the variables $(\tilde{A}^{k})^{i}$  are updated by  $(\tilde{A}^{k})^{i}=\tilde{A}^{k}$ if $i\in \mathcal{F}_{n}^{k}$ and $(\tilde{A}^{k})^{i}=(\tilde{A}^{k-1})^{i}$ otherwise.

The update scheme in \eqref{update_A_O} employs one-step inertial acceleration, which improves the global convergence of the iterates.  However, when the parameters $\alpha^k$, $\beta^k$, and $L_r$ are not carefully chosen, this method may not demonstrate significant acceleration compared to using the SAGA estimator without any acceleration. To better illustrate the effectiveness of the acceleration method,   the multi-step inertial scheme was proposed in \cite{Liang2016}, which can be given by
\[
\begin{aligned}
	&\tilde{x}^{k}=x^{k}+\sum_{i=1}^t \alpha^{k+1-i}(x^{k+1-i}-x^{k-i}), \\
	&\underline{x}^{k}=x^{k}+\sum_{i=1}^t \beta^{k+1-i}(x^{k+1-i}-x^{k-i}),\\
	& x^{k+1}= \tilde{x}^{k}-\eta^{k} \nabla f(\underline{x}^{k}).
\end{aligned}
\]
Figure \ref{multi-step} demonstrates that the multi-step inertial scheme offers greater flexibility and possibly achieves better acceleration than the one-step approach. 
Multi-step acceleration also allows even a negative inertial parameter, which performs better than all nonnegative parameters, but for one-step acceleration, a negative parameter is not a good choice.
\begin{figure}[!htb]
\setlength\tabcolsep{2pt}
\centering
\begin{tabular}{cccccc}
\includegraphics[width=0.7\textwidth]{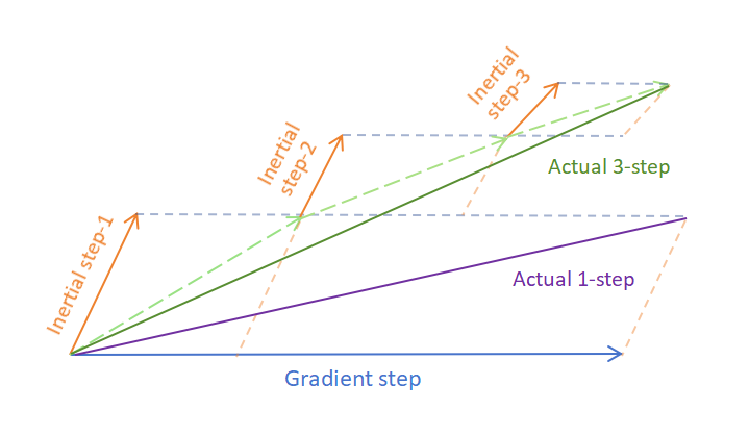}
\end{tabular}
\caption{The illustration of multi-step inertial acceleration.}
\label{multi-step}
\end{figure}


\section{Multi-step inertial accelerated doubly SGD algorithm}\label{algorithm}

In this section, we firstly introduce a multi-step inertial accelerated doubly stochastic gradient descent algorithm, a novel approach tailored for the rank-$(L_r, L_r, 1)$ block term decomposition problem \eqref{regular_LL1D}, denoted by Midas-LL1. This algorithm optimizes all partitioned matrices of every factor in parallel, using a gradient descent mechanism enhanced with a multi-step inertial framework to improve convergence. 

We first reformulate the subproblem \eqref{matrix_ALS} equivalently as 
\begin{eqnarray}\label{sum_matrix}
	{A}_{n,r}^{*}=\underset{A_{n,r}}{\arg\min}\quad\frac{1}{2I^3}\left\|X_{(n)}-\sum_{r=1}^R H_{n,r}A_{n,r}^{\top}\right\|_{F}^{2}+h_n(A_{n,r}), 
\end{eqnarray} 
where $\sum_{r=1}^R h_n(A_{n,r})=h_n(A_n)$. The gradient of $f$ with respect to partition-wise factor $A_{n,r}$ for $r=1, 2,\ldots, R$ is equal to 
\begin{equation}\label{btd_grad_f}
	\nabla_{A_{n, r}} f(A_{1},A_{2},A_{3})=\frac{1}{I^N}\left(\left(\sum_{r=1}^R H_{n,r}A_{n,r}^{\top}\right)^\top H_{n,r}-X_{(n)}^{\top}H_{n,r}\right)=\frac{1}{I^N}(A_{n}  H_{n}^{\top}H_{n,r}-X_{(n)}^{\top}H_{n,r}),
\end{equation}
where $H_{n,r}$ is the partition-wise Kronecker product in \eqref{H_12}. 
Let $\tilde{\nabla}_{A_{n,r}}f\in\mathbb{R}^{I_{n}\times L_r}$ be the stochastic gradient of $f(A_{1},A_{2},A_{3})$ for $A_{n,r}$, $n = 1, 2, 3$, then we have
	\begin{eqnarray}
		\begin{aligned}
			\tilde{\nabla}_{A_{n,r}}f(A_{1},A_{2},A_{3})&=\frac{1}{{I_n}\left|\mathcal{F}_{n}\right|}\left(A_{n} H_{n}^{\top}\left(\mathcal{F}_{n}\right) H_{n,r}\left(\mathcal{F}_{n}\right)-X_{n}^{\top}\left(\mathcal{F}_{n}\right) H_{n,r}\left(\mathcal{F}_{n}\right)\right), 
		\end{aligned}\label{sto_gradient}
	\end{eqnarray}
 where 
\[
X_{n}\left(\mathcal{F}_{n}\right)=X_{n}\left(\mathcal{F}_{n},:\right), \quad H_{n}\left(\mathcal{F}_{n}\right)=H_{n}\left(\mathcal{F}_{n},:\right).
\]
At each iteration, we choose one index $n$ from $\{1,\dots,N\}$ in order and randomly sample some fiber indexes $\mathcal F_n$ from $\{1,\dots,J_n\}$. Figure \ref{fiber} shows the fiber-sampling strategy for the stochastic gradient with batch size $\left|\mathcal{F}_{n}\right|=3$. 
\begin{figure}[!htb]
\setlength\tabcolsep{2pt}
\centering
\begin{tabular}{cccccc}
\includegraphics[width=0.8\textwidth]{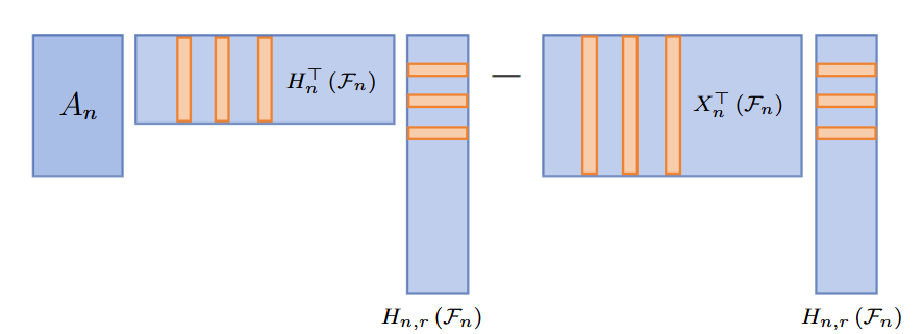}
\end{tabular}
\caption{Fiber-sampling strategy with batchsize $\left|\mathcal{F}_{n}\right|=3$.}
\label{fiber}
\end{figure}

We compute the multi-step inertial accelerated stochastic gradient estimator $\tilde{\nabla}_{A_{n,r}}f(\dots,\underline{A}_n^k,\dots)$, shortly denoted as  $\tilde{\nabla}_{A_{n,r}}f(\underline{A}_n^k)$. At last, we update $A_n^{k+1}$ by the stochastic proximal gradient descent method.

The algorithmic framework of Midas-LL1 (Algorithm \ref{ipSGD}) for optimization problem \eqref{regular_LL1D} is presented as follows. 

\begin{algorithm}[H]
	\caption{Midas-LL1: multi-step inertial accelerated doubly SGD for rank-$(L_r, L_r, 1)$ block-term decomposition problem \eqref{regular_LL1D}}
	\label{ipSGD}
	{\bfseries Input:} A third-order tensor $\mathcal{X}\in\mathbb{R}^{I_{1}\times I_{2} \times I_{3}}$; the rank of partitioned factors in vector $L=[L_1,\cdots,L_R]$; the number of components $R$;  the sample size $B$; initialization $\{A_{n}^{-1}\}_{n=1}^{3}=\{A_{n}^{0}\}_{n=1}^{3}$, setting stepsize $\{\eta^{k}\}_{k\ge0}$;  an integer $t \geq 1$, and inertial parameters $\{\alpha^{k+1-i}\}_{k\ge 0},\{\beta^{k+1-i}\}_{k\ge 0}\in[0,1]$ with $i \in \{1,\dots, t\}$. 
	\begin{algorithmic}[1]
		\State $k\leftarrow 0$;
		\Repeat 
		\State Sample $n$ uniformly  from $\{1, \dots, N\}$.
		\State Sample $\mathcal{F}_{n}$ uniformly from $\{1,\dots,J_{n}\}$ with $|\mathcal{F}_{n}|=B$.
		\State {Set  $\tilde{A}_n^k=A_n^k+\sum_{i=1}^t \alpha^{k+1-i}\left(A_n^{k+1-i}-A_n^{k-i}\right)$,  $\underline{A}_n^k=A_n^k+\sum_{i=1}^t \beta^{k+1-i}\left(A_n^{k+1-i}-A_n^{k-i}\right)$. }
		\State Compute  stochastic gradient $\tilde{\nabla}_{A_{n,r}}f(\underline{A}_n^k)$ with the minibatch  $B$ that satisfies Definition~\ref{vr_definition}. 
		\State Update partitioned matrices $A_{n,r}^{k+1}$ of $A_{n}^{k+1}$ for $r=1,\ldots,R$ and $A_{n^{\prime}}^{k+1}$:
		\begin{eqnarray}
		  \begin{aligned}\label{Anrk_update}
			A_{n,r}^{k+1} &= \underset{A_{n,r}}{\arg\min}\,\,h_{n}\left(A_{n,r}\right)+\langle\tilde{\nabla}_{A_{n,r}}f(\underline{A}_n^k), A_{n,r}-(\tilde{A}_{n,r})^{k}\rangle+\frac{1}{2\eta^{k}} \|A_{n,r}-(\tilde{A}_{n,r})^{k}\|_{F}^{2}, \\
			A_{n^{\prime}}^{k+1} &= A_{n^{\prime}}^{k}, \quad \forall n^{\prime} \neq n.
		\end{aligned}  
		\end{eqnarray}
		
		\State $k\leftarrow k+1$;
		\Until{some stopping criterion is reached;}
	\end{algorithmic}
	{\bfseries Output:} $\{A_{n}^{k}\}_{n=1}^{N}$.
\end{algorithm}

\begin{remark}
The Midas-LL1 serves as a general algorithm framework: 
\begin{itemize}
\item[(i)] If $L_r=1$ for $r=1,\ldots, R$, and $t=1$, then Midas-LL1 (Algorithm \ref{ipSGD}) is equivalent to iBrasCPD \cite{Wang2023}. 
\item[(ii)] If $L_r=1$ for $r=1,\ldots, R$, and $t=0$, then  Midas-LL1 algorithm  is similar to BrasCPD \cite{FuIWGH20} with vanilla SGD.

\end{itemize}
\end{remark}

\begin{assumption}\label{assume_ipsgd}
	\begin{itemize}
		\item[(i)] $h_{n}: \mathbb{R}^{I_{n}\times R}\rightarrow\mathbb{R}\cup \{+\infty\}$ are proper lower semi-continuous (l.s.c.) functions that are bounded from below, and $h(\cdot)$ is convex.
		\item[(ii)] The sequence $\{A_1^{k}, A_2^{k}, A_3^{k}\}$ generated by Algorithm \ref{ipSGD} is bounded for all $k$.
  \end{itemize}
  \end{assumption}
  
\begin{remark}\label{remark_03}
\begin{itemize}
\item[(i)] Suppose   Assumption \ref{assume_ipsgd}~(ii) holds, then for any factors   $\{A_1, A_2, A_3\}$,  $\{\underline A_{1}, \underline A_2,\underline A_{3}\}$ and any mode $n\in\{1, 2, 3\}$, there exists a constant $L$ such that
  \[
\begin{aligned}
\|\nabla_{{A}_{n}}f\left(\dots,\bar A_n,\dots\right)-\nabla_{A_{n}}f\left(\dots, A_n,\dots\right)\|_{F}
\le L\|\bar{A}_{n}-A_{n}\|_{F},
\end{aligned}
\]
which shows that 
\begin{eqnarray}
\begin{aligned}
|f(\dots, A_n, \dots)-f(\dots, \underline A_n, \dots)-\left<\nabla_{A_{n}}f(\dots, \underline A_n, \dots ), A_n-\underline {A}_{n}\right>| \le\frac{L}{2}\|A_n-\bar{A}_{n}\|^{2}_{F}.
\end{aligned}\label{L-lip}
\end{eqnarray}
In fact,     $f(A_{1}, A_2, A_{3})$  is quadratic with respect to $A_{n}$, for all $n\in \{1, 2, 3\}$, and the quadratic term is  $\frac{1}{2I^N}A_n^{\top}(H_n^{\top}H_n)A_n$. Hence,  we have  \eqref{L-lip} holds  with  
\[
L\ge \frac{1}{I^3}\max\{ \lambda_{\max}({H}_{n}^{\top}{H}_{n})\}
\]
for any $n\in\{1,2,3\}$. 
\item[(ii)] The optimal objective value  $\Phi^{*}$ of the objective function $\Phi(A_1, A_2, A_{3})$ defined by \eqref{regular_LL1D}   is finite under Assumption  \ref{assume_ipsgd}~(ii). 
\end{itemize}
\end{remark}

Let $\xi^k$ and $\zeta^k$  be the stochastic parameters for the block index and the stochastic gradient, respectively. Denote $\mathbb E_k[\cdot]=\mathbb E[\cdot |  \xi^k,\zeta^k]$ and $\mathbb E[\cdot]=\mathbb E[\cdot |\xi^0,\zeta^0,\dots]$.

\begin{definition}(Multi-step variance-reduced stochastic gradient)\label{vr_definition}
	We say a gradient estimator $\tilde{\nabla}_{A_{\xi^k}}f$ with $\xi^k$ randomly selected from the index set $\{1,2,3\}$ ,  is variance-reduced with constants $V_{1}, V_{2},V_{\Gamma}\ge 0$, $t\ge1$, and $\tau\in(0,1]$ if it satisfies the following conditions:
	\begin{itemize}
		\item[(i)] (Mean squared error (MSE) bound): there exists a sequence of random variables $\{\Gamma_{k}\}_{k\ge1}$ such that
		
			\begin{eqnarray}
				\begin{aligned}
	\mathbb{E}_{k}[\|\tilde{\nabla}_{A_{\xi^k}}f(\underline{A}_{\xi^{k}}^{k})-\nabla_{A_{\xi^k}}f(\underline{A}_{\xi^k}^{k})\|_{F}^{2}]
	\le\Gamma_{k}+V_1 \,\left(\sum_{i=1}^{t+1} \left\|A^{k+1-i}-A^{k-i}\right\|_{F}^2\right),
\end{aligned}\label{MSE_l22}
\end{eqnarray}

and  random variables $\{\Upsilon_{k}\}_{k\ge1}$ such that

\begin{eqnarray}
\begin{aligned}
\mathbb{E}_{k}[\|\tilde{\nabla}_{A_{\xi^k}}f(\underline{A}_{\xi^{k}}^{k})-\nabla_{A_{\xi^k}}f(\underline{A}_{\xi^{k}}^{k})\|_{F}]
\le\Upsilon_{k}+ V_2 \,\left(\sum_{i=1}^{t+1}\left\|A^{k+1-i}-A^{k-i}\right\|_{F}\right).
\end{aligned}\label{MSE_l2}
\end{eqnarray}

\item[(ii)] (Geometric decay): The sequence $\{\Gamma_{k}\}_{k\ge1}$ satisfy the following inequality in expectation:
\begin{eqnarray}
\begin{aligned}
\mathbb{E}_{k}[\Gamma_{k+1}]\le& (1-\tau)\Gamma_{k}+V_{\Gamma}\,\left(\sum_{i=1}^{t+1} \left\|A^{k+1-i}-A^{k-i}\right\|_{F}^2\right).\label{Gamma_k1_k}
\end{aligned}
\end{eqnarray}
\item[(iii)] (Convergence of estimator): For all sequences $\{A^{k}\}_{k=0}^{\infty}$, if  
$\lim_{k\rightarrow\infty}\mathbb{E}\left\|A^{k}-A^{k-1}\right\|_{F}^{2}\rightarrow 0$, then it follows that $\mathbb{E}\Gamma_{k}\rightarrow0$ and $\mathbb{E}\Upsilon_{k}\rightarrow0$.
\end{itemize}
\end{definition}

\section{Convergence analysis}\label{convergence_analysis}
This section is dedicated to the convergence analysis of the sequence generated by Algorithm \ref{ipSGD}.
\subsection{Subsequential convergence analysis}

We first present the   descent of $\Phi(A_1^{k+1}, A_2^{k+1}, A_3^{k+1})$ under expectation  in the following lemma.

\begin{lemma}\label{lemma_Phi_kk1}
Suppose  $\{A_1^{k}, A_2^{k}, A_3^{k}\}_{k\in\mathbb{N}}$ is the sequence generated by Algorithm \ref{ipSGD}. Assume that Assumption \ref{assume_ipsgd} holds  and  $\tilde{\nabla}_{A_{n,r}}f$ with $n=1,2,3$ is variance-reduced as defined in Definition \ref{vr_definition}. Define the following quantities for $k \in \mathbb{N}$,
\begin{equation}\label{bk}
  b_k \stackrel{\text { def }}{=} \frac{1-\sum_{i=1}^t\alpha^{k+1-i}-2L\eta^k -\gamma\eta^k }{2 \eta^k}, \quad
\underline{b} \stackrel{\text { def }}{=} \liminf _{k \in \mathbb{N}} b_k,  
\end{equation}
and
\begin{equation}\label{aki}
 a_{k, i} \stackrel{\text { def }}{=} \frac{3Lt\left(\beta^{k+1-i}\right)^2}{2}+\frac{\gamma}{2}+\frac{\alpha^{k+1-i}}{2\eta^{k}}, \quad \bar{a}_i \stackrel{\text { def }}{=} \limsup _{k \in \mathbb{N}} a_{k, i}.
\end{equation}
Here,  $\gamma=\sqrt{(V_{\Gamma}/\tau+V_{1})}$,  $V_{1}, V_{2},V_{\Gamma}\ge 0$, $\tau\in(0,1]$ are parameters in Definition~\ref{vr_definition}, and $\{\alpha^{k+1-i}\}_{k\ge 0}$, $\{\beta^{k+1-i}\}_{k\ge 0}\in[0,1]$ are inertial parameters in Algorithm \ref{ipSGD}. Then the following inequality holds for any $k>0$, 
\begin{eqnarray*}
\begin{aligned}
&\mathbb{E}_{k}[\Phi(A_1^{k+1}, A_2^{k+1}, A_3^{k+1})]+\underline{b} \,\mathbb{E}_{k}\left[\left\|A^{k+1}-A^k\right\|_{F}^2\right]+\frac{1}{2\gamma\tau}\mathbb{E}_{k}[\Gamma_{k+1}] \\
\leq & \Phi(A_1^{k}, A_2^{k}, A_3^{k})+\frac{1}{2\gamma\tau}\Gamma_{k}+\sum_{i=1}^t \bar a_{ i}\left\|A^{k+1-i}-A^{k-i}\right\|_{F}^2+\frac{\gamma}{2} \left\|A^{k-t}-A^{k-t-1}\right\|_{F}^2.
\end{aligned}
\end{eqnarray*}
\end{lemma}

\begin{proof}
Assume $n=\xi^{k}$ at the $k$-th iteration. Based on the update rule for $A_{\xi^k,r}^{k+1}$ in Midas-LL1 framework (Algorithm \ref{ipSGD}), we establish that
\[
\begin{aligned}
A_{\xi^k,r}^{k+1} &= \underset{A_{\xi^k,r}}{\arg \min}\,\,h_{\xi^k}\left(A_{\xi^k,r}\right)+\langle\tilde{\nabla}_{A_{\xi^k,r}}f(\underline{A}_n^k), A_{\xi^k,r}-\tilde{A}_{\xi^k,r}^{k}\rangle+\frac{1}{2\eta^{k}} \|A_{\xi^k,r}-\tilde{A}_{\xi^k,r}^{k}\|_{F}^{2}. 
\end{aligned}
\] 
The inequality below holds by definition of $A_{\xi^{k},r}^{k+1}$
	\begin{eqnarray}\label{h}
	    \begin{aligned}
		&\sum_{r=1}^R\left(h_{\xi^{k}}(A_{\xi^k,r}^{k+1})+\left\langle \tilde{\nabla}_{A_{\xi^k,r}}f(\underline{A}_{\xi^{k}}^k), A_{\xi^k,r}^{k+1}-\tilde A_{\xi^k,r}^{k}\right\rangle +\frac{1}{2\eta^{k}}\left\|A_{\xi^k,r}^{k+1}-\tilde A_{\xi^k,r}^{k}\right\|_{F}^{2}\right)\\
		\le&\sum_{r=1}^R\left(h_{\xi^{k}}(A_{\xi^k,r}^{k})+\left\langle \tilde{\nabla}_{A_{\xi^k,r}}f(\underline{A}_{\xi^{k}}^k),A_{\xi^k,r}^{k}-\tilde A_{\xi^k,r}^{k}\right\rangle +\frac{1}{2\eta^{k}}\left\|A_{\xi^k,r}^{k}-\tilde A_{\xi^k,r}^{k}\right\|_{F}^{2}\right).
	\end{aligned}
	\end{eqnarray}
According to  \eqref{L-lip}, we derive
$$
\begin{aligned}
f(A_1^{k+1}, A_2^{k+1}, A_3^{k+1}) &\leq f( \cdots, \underline{A}_{\xi^{k}}^k, \cdots )\\
&+\sum_{r=1}^R\left\langle{\nabla}_{A_{\xi^k,r}}f(\underline{A}_{\xi^{k}}^k), A_{\xi^k,r}^{k+1}-\underline A_{\xi^k,r}^{k}\right\rangle +\sum_{r=1}^R\frac{L}{2}\left\|A_{\xi^k,r}^{k+1}-\underline A_{\xi^k,r}^{k}\right\|_{F}^{2},
\end{aligned}
$$
and
$$
\begin{aligned}
      f(A_1^{k}, A_2^{k}, A_3^{k}) &\geq f( \cdots, \underline{A}_{\xi^{k}}^k,\cdots )\\
      &+ \sum_{r=1}^R\left\langle{\nabla}_{A_{\xi^k,r}}f(\underline{A}_{\xi^{k}}^k), A_{\xi^k,r}^{k}-\underline A_{\xi^k,r}^{k}\right\rangle -\sum_{r=1}^R\frac{L}{2}\left\|A_{\xi^k,r}^{k}-\underline A_{\xi^k,r}^{k}\right\|_{F}^{2}.
\end{aligned}
$$
Combining two inequalities about $f$, we can get
\begin{eqnarray}\label{f}
\begin{aligned}
f(A_1^{k+1}, A_2^{k+1}, A_3^{k+1}) &\leq   f(A_1^{k}, A_2^{k}, A_3^{k})+\sum_{r=1}^R\left\langle{\nabla}_{A_{\xi^k,r}}f(\underline{A}_{\xi^{k}}^k), A_{\xi^k,r}^{k+1}-A_{\xi^k,r}^{k}\right\rangle\\
&+\sum_{r=1}^R\frac{L}{2}\left\|A_{\xi^k,r}^{k+1}-\underline A_{\xi^k,r}^{k}\right\|_{F}^{2}+\sum_{r=1}^R\frac{L}{2}\left\|A_{\xi^k,r}^{k}-\underline A_{\xi^k,r}^{k}\right\|_{F}^{2}.
\end{aligned} 
\end{eqnarray}

By summing inequalities \eqref{h} and  \eqref{f}  together, we derive
\begin{eqnarray}
\begin{aligned}
&\Phi(A_1^{k+1}, A_2^{k+1}, A_3^{k+1}) \\
 \leq& \Phi(A_1^{k}, A_2^{k}, A_3^{k}) + \sum_{r=1}^R\left\langle{\nabla}_{A_{\xi^k,r}}f(\underline{A}_{\xi^{k}}^k)-\tilde{\nabla}_{A_{\xi^k,r}}f(\underline{A}_{\xi^{k}}^k), A_{\xi^k,r}^{k+1}-A_{\xi^k,r}^{k}\right\rangle\\
 &+\sum_{r=1}^R\frac{L}{2}\left\|A_{\xi^k,r}^{k+1}-\underline A_{\xi^k,r}^{k}\right\|_{F}^{2}+\sum_{r=1}^R\frac{L}{2}\left\|A_{\xi^k,r}^{k}-\underline A_{\xi^k,r}^{k}\right\|_{F}^{2}\\
 &+\sum_{r=1}^R\frac{1}{2\eta^{k}}\left\|A_{\xi^k,r}^{k}-\tilde A_{\xi^k,r}^{k}\right\|_{F}^{2}-\sum_{r=1}^R\frac{1}{2\eta^{k}}\left\|A_{\xi^k,r}^{k+1}-\tilde A_{\xi^k,r}^{k}\right\|_{F}^{2}\\
 \leq& \Phi(A_1^{k}, A_2^{k}, A_3^{k}) + \frac{1}{2\gamma^{k}}  \left\|\tilde{\nabla}_{A_{\xi^k}}f(\underline{A}_{\xi^{k}}^{k})-\nabla_{A_{\xi^k}}f(\underline{A}_{\xi^k}^{k})\right\|_{F}^{2}+ \frac{\gamma^{k}}{2}\left\|A_{\xi^k}^{k+1}-A_{\xi^k}^{k}\right\|_{F}^2\\
 &+\frac{3L}{2}\left\|\sum_{i=1}^t \beta^{k+1-i}(A_{\xi^k}^{k+1-i}-A_{\xi^k}^{k-i})\right\|_{F}^2+\frac{1}{2 \eta^k}  \left\|\sum_{i=1}^t \alpha^{k+1-i}(A_{\xi^k}^{k+1-i}-A_{\xi^k}^{k-i})\right\|_{F}^2\\
 &+ L\left\|A_{\xi^k}^{k+1}-A_{\xi^k}^{k}\right\|_{F}^2-\frac{1}{2\eta^{k}}\left\|A_{\xi^k}^{k+1}-\tilde A_{\xi^k}^{k}\right\|_{F}^{2}\\
 \leq& \Phi(A_1^{k}, A_2^{k}, A_3^{k}) + \frac{1}{2\gamma^{k}}  \left\|\tilde{\nabla}_{A_{\xi^k}}f(\underline{A}_{\xi^{k}}^{k})-\nabla_{A_{\xi^k}}f(\underline{A}_{\xi^k}^{k})\right\|_{F}^{2}+ \frac{\gamma^{k}+2L}{2}\left\|A_{\xi^k}^{k+1}-A_{\xi^k}^{k}\right\|_{F}^2\\
 &+\frac{3L}{2}\sum_{i=1}^t t\left(\beta^{k+1-i}\right)^2\left\|A_{\xi^k}^{k+1-i}-A_{\xi^k}^{k-i}\right\|_{F}^2+\frac{1}{2 \eta^k}  \left\|\sum_{i=1}^t \alpha^{k+1-i}(A_{\xi^k}^{k+1-i}-A_{\xi^k}^{k-i})\right\|_{F}^2\\
 &-\frac{1}{2\eta^{k}}\left\|A_{\xi^k}^{k+1}-\tilde A_{\xi^k}^{k}\right\|_{F}^{2}\\
 \leq&\Phi(A_1^{k}, A_2^{k}, A_3^{k}) +  \frac{1}{2\gamma^{k}}  \left\|\tilde{\nabla}_{A_{\xi^k}}f(\underline{A}_{\xi^{k}}^{k})-\nabla_{A_{\xi^k}}f(\underline{A}_{\xi^k}^{k})\right\|_{F}^{2}+ \frac{\gamma^{k}+2L}{2}\left\|A_{\xi^k}^{k+1}-A_{\xi^k}^{k}\right\|_{F}^2\\
 &+ \frac{3L}{2}\sum_{i=1}^t t\left(\beta^{k+1-i}\right)^2\left\|A_{\xi^k}^{k+1-i}-A_{\xi^k}^{k-i}\right\|_{F}^2+\sum_{i=1}^t \frac{\alpha^{k+1-i}}{2\eta^{k}}\left\|A_{\xi^k}^{k+1-i}-A_{\xi^k}^{k-i}\right\|_{F}^{2}\\
 &-\frac{(1-\sum_{i=1}^t\alpha^{k+1-i})}{2\eta^{k}}\left\|A_{\xi^k}^{k+1}-A_{\xi^k}^{k}\right\|_{F}^{2},\\
\end{aligned}\label{ineq_01}
\end{eqnarray}
where $\gamma^{k}>0$ is any constant, the second inequality is derived from Young's inequality  $\langle a,b\rangle \le \frac{1}{2c}\|a\|^{2}+\frac{c}{2}\|b\|^{2}$ for $a, b\in\mathbb{R}^n, c\in\mathbb{R}_{++}$. The third inequality is obtained from 
\begin{equation*}
    \left\|A_{\xi^k}^{k}-\underline A_{\xi^k}^{k}\right\|_{F}^{2}=\left\|\sum_{i=1}^t \beta^{k+1-i}(A_{\xi^k}^{k+1-i}-A_{\xi^k}^{k-i})\right\|_{F}^2\leq \sum_{i=1}^t t\left(\beta^{k+1-i}\right)^2\left\|A_{\xi^k}^{k+1-i}-A_{\xi^k}^{k-i}\right\|_{F}^2,
\end{equation*}
and the last inequality is deduced from
        \[
	\begin{aligned}
        &\left\|A_{\xi^k}^{k+1}-\tilde A_{\xi^k}^{k}\right\|_{F}^{2}\\=
		&\left\|A_{\xi^k}^{k+1}-A_{\xi^k}^{k}-\sum_{i=1}^t \alpha^{k+1-i}(A_{\xi^k}^{k+1-i}-A_{\xi^k}^{k-i})\right\|_{F}^{2}\\
		=&\left\|A_{\xi^k}^{k+1}-A_{\xi^k}^{k}\right\|_{F}^{2}-2\sum_{i=1}^t \alpha^{k+1-i}\left\langle A_{\xi^k}^{k+1}-A_{\xi^k}^{k},A_{\xi^k}^{k+1-i}-A_{\xi^k}^{k-i}\right\rangle + \left\|\sum_{i=1}^t \alpha^{k+1-i}(A_{\xi^k}^{k+1-i}-A_{\xi^k}^{k-i})\right\|_{F}^{2}\\
		\geq &(1-\sum_{i=1}^t\alpha^{k+1-i})\left\|A_{\xi^k}^{k+1}-A_{\xi^k}^{k}\right\|_{F}^{2}
        +\left\|\sum_{i=1}^t \alpha^{k+1-i}(A_{\xi^k}^{k+1-i}-A_{\xi^k}^{k-i})\right\|_{F}^{2}-\sum_{i=1}^t \alpha^{k+1-i}\left\|A_{\xi^k}^{k+1-i}-A_{\xi^k}^{k-i}\right\|_{F}^{2}.
	\end{aligned}
	\]
 Applying the conditional expectation operator $\mathbb{E}_{k}$ to the inequality \eqref{ineq_01} and constraining the mean squared error  term by \eqref{MSE_l22} as specified in Definition \ref{vr_definition}, we obtain
\begin{eqnarray}
\begin{aligned}
&\mathbb{E}_{k}[\Phi(A_1^{k+1}, A_2^{k+1}, A_3^{k+1})] \\
\leq & \Phi(A_1^{k}, A_2^{k}, A_3^{k})+\frac{1}{2\gamma^{k}}\mathbb{E}_{k}[\|\tilde{\nabla}_{A_{\xi^k}}f(\underline{A}_{\xi^{k}}^{k})-\nabla_{A_{\xi^k}}f(\underline{A}_{\xi^k}^{k})\|_{F}^{2}]+\frac{\gamma^{k}+2L}{2}\mathbb{E}_{k}[\|A^{k+1}-A^k\|_{F}^2]\\
 &+ \frac{3L}{2} \sum_{i=1}^t t\left(\beta^{k+1-i}\right)^2\|A^{k+1-i}-A^{k-i}\|_{F}^2+\sum_{i=1}^t \frac{\alpha^{k+1-i}}{2\eta^{k}}\|A^{k+1-i}-A^{k-i}\|_{F}^{2}\\
  &-\frac{(1-\sum_{i=1}^t\alpha^{k+1-i})}{2\eta^{k}}\mathbb{E}_{k}[\|A^{k+1}-A^k\|_{F}^2]\\
\leq & \Phi(A_1^{k}, A_2^{k}, A_3^{k})+\frac{1}{2\gamma^{k}\tau}(\Gamma_{k}-\mathbb{E}_{k}[\Gamma_{k+1}])+\left(\frac{\gamma^{k}+2L}{2}-\frac{(1-\sum_{i=1}^t\alpha^{k+1-i})}{2\eta^k}\right)\mathbb{E}_{k}[\|A^{k+1}-A^k\|_{F}^2]\\
&+\left(\frac{V_{\Gamma}}{2\gamma^{k}\tau }+\frac{V_{1}}{2\gamma^{k}}\right)(\sum_{i=1}^{t+1} \|A^{k+1-i}-A^{k-i}\|_{F}^2)+ \frac{3L}{2}\sum_{i=1}^t t\left(\beta^{k+1-i}\right)^2\|A^{k+1-i}-A^{k-i}\|_{F}^2\\
&+\sum_{i=1}^t \frac{\alpha^{k+1-i}}{2\eta^{k}}\|A^{k+1-i}-A^{k-i}\|_{F}^{2}.
\end{aligned}\label{ineq_002}
\end{eqnarray}
Let  $\gamma=\gamma^{k}=\sqrt{(V_{\Gamma}/\tau+V_{1})}$, then we can get

\begin{eqnarray*}
\begin{aligned}
&\mathbb{E}_{k}[\Phi(A_1^{k+1}, A_2^{k+1}, A_3^{k+1})]+\left(\frac{(1-\sum_{i=1}^t\alpha^{k+1-i})}{2\eta^k}-\frac{\gamma+2L}{2}\right)\mathbb{E}_{k}[\|A^{k+1}-A^k\|_{F}^2]+\frac{1}{2\gamma\tau}\mathbb{E}_{k}[\Gamma_{k+1}] \\
\leq & \Phi(A_1^{k}, A_2^{k}, A_3^{k})+\frac{1}{2\gamma\tau}\Gamma_{k}+\sum_{i=1}^t \left(\frac{3Lt\left(\beta^{k+1-i}\right)^2}{2}+\frac{\gamma}{2}+\frac{\alpha^{k+1-i}}{2\eta^{k}}\right)\|A^{k+1-i}-A^{k-i}\|_{F}^2\\
&+\frac{\gamma}{2} \|A^{k-t}-A^{k-t-1}\|_{F}^2.
\end{aligned}
\end{eqnarray*}

Denote $b_k$ and $a_{k,i}$ as in \eqref{bk} and \eqref{aki}, respectively.
Consequently, we obtain

\begin{eqnarray*}
\begin{aligned}
&\mathbb{E}_{k}[\Phi(A_1^{k+1}, A_2^{k+1}, A_3^{k+1})]+\underline{b}\, \mathbb{E}_{k}[\|A^{k+1}-A^k\|_{F}^2]+\frac{1}{2\gamma\tau}\mathbb{E}_{k}[\Gamma_{k+1}] \\
\leq &\mathbb{E}_{k}[\Phi(A_1^{k+1}, A_2^{k+1}, A_3^{k+1})]+b_{k}\mathbb{E}_{k}[\|A^{k+1}-A^k\|_{F}^2]+\frac{1}{2\gamma\tau}\mathbb{E}_{k}[\Gamma_{k+1}]\\
\leq & \Phi(A_1^{k}, A_2^{k}, A_3^{k})+\frac{1}{2\gamma\tau}\Gamma_{k}+\sum_{i=1}^t a_{k, i}\|A^{k+1-i}-A^{k-i}\|_{F}^2+\frac{\gamma}{2} \|A^{k-t}-A^{k-t-1}\|_{F}^2\\
\leq & \Phi(A_1^{k}, A_2^{k}, A_3^{k})+\frac{1}{2\gamma\tau}\Gamma_{k}+\sum_{i=1}^t \bar a_{ i}\|A^{k+1-i}-A^{k-i}\|_{F}^2+\frac{\gamma}{2} \|A^{k-t}-A^{k-t-1}\|_{F}^2.
\end{aligned}
\end{eqnarray*}
Hence, the proof is completed. 
\end{proof}

Following this, a new Lyapunov function is introduced. We prove that it is monotonically nonincreasing in expectation. To simplify notation, we denote $\Phi^k=\Phi\left(A_1^{k}, A_2^{k}, A_3^{k}\right)$.

\begin{lemma}\label{lyapunov_descent}
	    Assume that Assumption \ref{assume_ipsgd} holds and $\tilde{\nabla}_{A_{n,r}}f$ for $n=1,2,3$ is variance-reduced as defined in Definition \ref{vr_definition}. Let $\{A_1^{k}, A_2^{k}, A_3^{k}\}_{k \in \mathbb{N}}$ be a sequence generated by  Algorithm \ref{ipSGD}. We then define the following Lyapunov sequence
         \begin{eqnarray}
         \begin{aligned}
        \Psi_{k}:&= \Phi^{k}  +\sum_{i=1}^{t+1} \sum_{j=i}^{t+1}(j+1)\bar a_{j}\left\|A^{k+1-i}-A^{k-i}\right\|_{F}^2+\frac{1}{2\tau \gamma}\Gamma_{k}, \label{lyapunov_function}
        \end{aligned}
        \end{eqnarray}
       and choose $\eta^k$, $\gamma$, $\{\alpha^{k+1-i}\}_{k\ge0}$, and $\{\beta^{k+1-i}\}_{k\ge0}$ in \eqref{bk} and \eqref{aki} such that
      \begin{equation}\label{paraofLya}
           \delta = \underline{b}-\sum_{j=1}^{t+1}(j+1)\bar a_{j} > 0,\quad \mathrm{ and }\quad  (t+2) \bar{a}_{t+1}-\frac{\gamma}{2} > 0.
       \end{equation}
       Then, for all $k\in\mathbb{N}$, we have
       \begin{eqnarray}
       \begin{aligned}
       \mathbb{E}_{k}[\Psi_{k+1}]\le\Psi_{k} - \rho \left(\mathbb{E}_{k}[\|A^{k+1}-A^k\|_{F}^2]+\sum_{i=1}^{t+1} \left\|A^{k+1-i}-A^{k-i}\right\|_{F}^2\right),
      \end{aligned}\label{decent_inequality_01}
      \end{eqnarray} 
      where $\rho:=\min\{\delta, (t+2)\bar{a}_{t+1}-\frac{\gamma}{2}, \bar{a}\}$.
\end{lemma}
\begin{proof}
According to Lemma \ref{lemma_Phi_kk1}, we have
\begin{eqnarray}
\begin{aligned}
&\mathbb{E}_{k}[\Phi^{k+1}]-\Phi^{k}\\
\le& \, \sum_{i=1}^t \bar a_{ i}\|A^{k+1-i}-A^{k-i}\|_{F}^2+\frac{\gamma}{2} \|A^{k-t}-A^{k-t-1}\|_{F}^2-\underline{b}\, \mathbb{E}_{k}[\|A^{k+1}-A^k\|_{F}^2]\\
&-\frac{1}{2\gamma\tau}(\mathbb{E}_{k}[\Gamma_{k+1}]-\Gamma_{k}).\label{inequality_01}
\end{aligned}
\end{eqnarray}
Furthermore, by using \eqref{lyapunov_function} and integrating it  with \eqref{inequality_01}, we obtain
\begin{align*}
&\Psi_{k}-\mathbb{E}_{k}[\Psi_{k+1}]\\
=&\sum_{i=1}^{t+1} \sum_{j=i}^{t+1}(j+1)\bar a_{j}\|A^{k+1-i}-A^{k-i}\|_{F}^2-\sum_{i=1}^{t+1} \sum_{j=i}^{t+1}(j+1)\bar a_{j}\mathbb{E}_{k}\|A^{k+2-i}-A^{k+1-i}\|_{F}^2+\Phi^{k}-\mathbb{E}_{k}[\Phi_{k+1}]\\
&+\frac{1}{2\tau \gamma}(\Gamma_{k}-\mathbb{E}_{k}[\Gamma_{k+1}])\\
\geq&  \, \underline{b}\, \mathbb{E}_{k}[\|A^{k+1}-A^k\|_{F}^2]-\sum_{i=1}^t \bar a_{ i}\|A^{k+1-i}-A^{k-i}\|_{F}^2+\left((t+2)\bar{a}_{t+1}-\frac{\gamma}{2}\right) \|A^{k-t}-A^{k-t-1}\|_{F}^2\\
&+\sum_{i=1}^t \sum_{j=i}^{t+1} (j+1)\bar{a}_j\left\|A^{k+1-i}-A^{k-i}\right\|_{F}^2-\sum_{j=1}^{t+1} (j+1)\bar{a}_j \mathbb{E}_{k}[\left\|A^{k+1}-A^{k}\right\|_{F}^2]\\
&-\sum_{i=2}^{t+1} \sum_{j=i}^{t+1} (j+1)\bar{a}_j\left\|A^{k+2-i}-A^{k+1-i}\right\|_{F}^2\\
\geq&  \, (\underline{b}-\sum_{j=1}^{t+1}(j+1) \bar{a}_j) \mathbb{E}_{k}[\|A^{k+1}-A^k\|_{F}^2]+\left((t+2)\bar{a}_{t+1}-\frac{\gamma}{2}\right)  \|A^{k-t}-A^{k-t-1}\|_{F}^2\\
&+\sum_{i=1}^t \sum_{j=i}^{t+1} \bar{a}_j\left\|A^{k+1-i}-A^{k-i}\right\|_{F}^2-\sum_{i=1}^t \bar a_{ i}\|A^{k+1-i}-A^{k-i}\|_{F}^2-\sum_{i=2}^{t+1} \sum_{j=i}^{t+1} \bar{a}_j\left\|A^{k+2-i}-A^{k+1-i}\right\|_{F}^2\\
&+\sum_{i=1}^t \sum_{j=i}^{t+1} j \bar{a}_j\left\|A^{k+1-i}-A^{k-i}\right\|_{F}^2 -\sum_{i=2}^{t+1} \sum_{j=i}^{t+1} j \bar{a}_j\left\|A^{k+2-i}-A^{k+1-i}\right\|_{F}^2\\
\geq&\, (\underline{b}-\sum_{j=1}^{t+1}(j+1) \bar{a}_j) \mathbb{E}_{k}[\|A^{k+1}-A^k\|_{F}^2]+\left((t+2)\bar{a}_{t+1}-\frac{\gamma}{2}\right) \|A^{k-t}-A^{k-t-1}\|_{F}^2\\
&+\sum_{i=1}^t i \bar{a}_i\left\|A^{k+1-i}-A^{k-i}\right\|_{F}^2\\
\geq&\, \delta \,\mathbb{E}_{k}[\|A^{k+1}-A^k\|_{F}^2]+\left((t+2)\bar{a}_{t+1}-\frac{\gamma}{2}\right) \|A^{k-t}-A^{k-t-1}\|_{F}^2 + \blue{ \bar{a}}\sum_{i=1}^t \left\|A^{k+1-i}-A^{k-i}\right\|_{F}^2,
\end{align*}
where the last inequality follows from \eqref{paraofLya}, and   $\bar a= \min_{i \in \{1,\dots, t\}}{\bar{a}_i}$. The third inequality stems from the fact that
\begin{align*}
&\sum_{i=1}^t \bar a_{ i}\|A^{k+1-i}-A^{k-i}\|_{F}^2+\sum_{i=2}^{t+1} \sum_{j=i}^{t+1} \bar{a}_j\left\|A^{k+2-i}-A^{k+1-i}\right\|_{F}^2\\
=&\sum_{i=1}^t \bar a_{ i}\|A^{k+1-i}-A^{k-i}\|_{F}^2+\sum_{i=1}^{t} \sum_{j=i+1}^{t+1} \bar{a}_j\left\|A^{k+1-i}-A^{k-i}\right\|_{F}^2=\sum_{i=1}^t \sum_{j=i}^{t+1} \bar{a}_j\left\|A^{k+1-i}-A^{k-i}\right\|_{F}^2,
\end{align*}
and
\begin{align*}
&\sum_{i=1}^t \sum_{j=i}^{t+1} j \bar{a}_j\left\|A^{k+1-i}-A^{k-i}\right\|_{F}^2 -\sum_{i=2}^{t+1} \sum_{j=i}^{t+1} j \bar{a}_j\left\|A^{k+2-i}-A^{k+1-i}\right\|_{F}^2\\
=&\sum_{i=1}^t \sum_{j=i}^{t+1} j \bar{a}_j\left\|A^{k+1-i}-A^{k-i}\right\|_{F}^2 -\sum_{i=1}^{t} \sum_{j=i+1}^{t+1} j \bar{a}_j\left\|A^{k+1-i}-A^{k-i}\right\|_{F}^2=\sum_{i=1}^t i \bar{a}_i\left\|A^{k+1-i}-A^{k-i}\right\|_{F}^2.
\end{align*}
Denote $\rho:=\min\{\delta, (t+2)\bar{a}_{t+1}-\frac{\gamma}{2}, \bar{a}\}$. Then we have
\begin{align*}
\Psi_{k}-\mathbb{E}_{k}[\Psi_{k+1}]
\geq\, \rho \,(\mathbb{E}_{k}[\|A^{k+1}-A^k\|_{F}^2]+\sum_{i=1}^t \left\|A^{k+1-i}-A^{k-i}\right\|_{F}^2+ \|A^{k-t}-A^{k-t-1}\|_{F}^2).
\end{align*}
Therefore, the proof is completed.
\end{proof}
\begin{remark}
 Given that $\eta^k$ lies in $[\underline \eta, \bar\eta]$. From  \eqref{bk}, \eqref{aki}, and \eqref{paraofLya}, the following can be concluded:
    \begin{itemize}
        \item[(i)]  When $\alpha^{k+1-i}=\beta^{k+1-i}\equiv 0$ in \eqref{bk} and \eqref{aki}, condition \eqref{paraofLya} holds as long as $\bar\eta < \frac{2}{4L+\gamma(t+2)(t+3)}$.
        \item[(ii)] Set $t=1$. If $\alpha^{k}\equiv \alpha$, $\beta^{k}\equiv \beta$ (i.e. constant inertial parameters), then   \eqref{paraofLya} implies that $\alpha$, $\beta$ must lie within a certain region defined by the inequality,
        
        $$
        L \beta^2<\frac{1-2L \bar{\eta}-6 \gamma \bar{\eta}-6 \alpha}{15\, \bar{\eta}}.
        $$
        
        \item[(iii)]
        When $t \geq 2$, for each $i \in \{1,\dots, t\}$, let $\alpha^{k+1-i}\equiv \alpha^{i}$, $\beta^{k+1-i}\equiv \beta^{i}$, then \eqref{paraofLya} tells us that the $\beta^{i}$ must live in a region related to $\alpha^{i}$,
        $$
        \frac{3Lt}{2}  \sum_{i=1}^{t+1} (i+1)(\beta^i)^2  < \frac{1-\sum_{j=1}^t(i+2) \alpha^{i}-(t+2) \alpha^{t+1}-2L\bar{\eta}}{2 \bar{\eta}}-\frac{\gamma(t+2)(t+3)}{4}.
        $$

    \end{itemize}
\end{remark}
As indicated by Lemma \ref{lyapunov_descent}, $\Psi_{k+1}$ becomes nonincreasing in expectation if the stepsize $\eta^k$, the inertial parameters $\{\alpha^{k+1-i}\}_{k\ge0}$,  and $\{\beta^{k+1-i}\}_{k\ge0}$ are appropriately selected. Based on this result, we demonstrate our main conclusion in this subsection.

\begin{theorem}\label{subsequence_convergence}
Given the sequence $\{A_1^{k}, A_2^{k}, A_3^{k}\}_{k \in \mathbb{N}}$ generated by the Algorithm \ref{ipSGD}, the following conclusions can be drawn.
\begin{itemize}
\item[(i)] The sequence $\{\mathbb{E}[\Psi_{k}]\}_{k\in\mathbb{N}}$ is  monotonically nonincreasing.
\item[(ii)] $\sum\limits_{k=1}^{+\infty}\mathbb{E}\left[\left\|A^{k+1}-A^k\right\|_{F}^2\right]<+\infty$,
indicating that $\mathbb{E}\left[\left\|A^{k+1}-A^k\right\|_F^2\right] \rightarrow 0 \text { as } k \rightarrow \infty$.   

\item[(iii)]  For any positive integer $K$, we have $\min\limits_{1\le k\le K}\mathbb{E}\left[\left\|A^{k+1}-A^k\right\|_{F}^2\right]\le \frac{\Psi_{1}}{\epsilon K}$.
\end{itemize}
\end{theorem}
\begin{proof}
	\begin{itemize}
		\item[(i)] This statement follows directly from Lemma \ref{lyapunov_descent}, which ensures the nonincreasing nature of the sequence $\{\mathbb{E}[\Psi_{k}]\}_{k\in\mathbb{N}}$ since $\epsilon>0$.
		\item[(ii)] Accumulating   \eqref{decent_inequality_01} from $k=1$ to a positive integer $K$, we obtain
		\[
		\sum_{k=1}^{K}\mathbb{E}[\|A^{k+1}-A^k\|_{F}^2]\le \frac{1}{\epsilon}\mathbb{E}[\Psi_{1}-\Psi_{K+1}]\le \frac{1}{\epsilon}\Psi_{1},
		\]
		where the last inequality follows from the non-negativity of $\Psi_{k}$ for any $k>0$. Taking the limit as $K\rightarrow+\infty$, we have
        $$
        \sum_{k=1}^{+\infty}\mathbb{E}[\|A^{k+1}-A^k\|_{F}^2]<+\infty.
        $$
        Therefore, it follows that the sequence  $\{\mathbb{E}[ \|A^{k+1}-A^k\|_{F}^2]\}$ converges to zero. 
		\item[(iii)] For any positive integer $K$, the following inequality is satisfied
		\[
		K\min_{1\le k\le K}\mathbb{E}[\|A^{k+1}-A^k\|_{F}^2]\le\sum_{k=1}^{K}\mathbb{E}[\|A^{k+1}-A^k\|_{F}^2]\le\frac{1}{\epsilon}\Psi_{1},
		\]
		which yields the desired result. 
	\end{itemize}
	This completes the proof.
\end{proof}

\subsection{Sequential convergence analysis}

In this subsection, we explore the sequential convergence of the Midas-LL1 algorithm, demonstrating the convergence of the entire sequence to an  $\epsilon$-stationary point.

\begin{lemma} \label{subgradient_bound}
Suppose that Assumption \ref{assume_ipsgd} holds, the stepsize $\eta^k$ lies in $[\underline \eta, \bar\eta]$, \eqref{paraofLya} is satisfied, and  assume that the sequences 
$\{\alpha^{k+1-i}\}_{i=1}^t$ and $\{\beta^{k+1-i}\}_{i=1}^t$ are both non-decreasing with limits $\lim_{k \rightarrow \infty}\alpha_{k}=c_{1}$ and $\lim_{k \rightarrow \infty}\beta_{k}=c_{2}$.  Let $\{A_1^{k}, A_2^{k}, A_3^{k}\}_{k\in\mathbb{N}}$ be a bounded sequence generated by Algorithm \ref{ipSGD}. Define
\[
 P_{\xi^k}^{k+1}:=\nabla_{A_{\xi^k}}f(A^{k+1})-\tilde{\nabla}_{A_{\xi^k}}f(\underline{A}^k)+\frac{1}{\eta^k}(\tilde{A}_{\xi^k}^{k}-A_{\xi^k}^{k+1}), 
\]
where $\underline A^{k}:=(\dots, \underline{A}_{n}^{k},\dots)$.  
Then we have $P_{\xi^k}^{k+1}\in \partial_{\xi^k} \Phi\left(A^{k+1}\right)$. Furthermore, denote $P^{k+1}=\left(P_{1}^{k+1}, P_{2}^{k+1}, P_{3}^{k+1}\right)$, where $P_{i}^{k+1}=P_{\xi^k}^{k+1}$ if $i=\xi^k$ and $P_{i}^{k+1}=0$ otherwise. 
 Then  { $3\mathbb E_k[P^{k+1}] \in \partial \Phi\left(A^{k+1}\right)$}, and we can obtain the following results
\[
\mathbb{E}_{k}\|P^{k+1}\|_{F} \le \mu\left(\mathbb{E}_{k}[\|A^{k+1}-A^k\|_{F}]+\sum_{i=1}^{t+1} \|A^{k+1-i}-A^{k-i}\|_{F}\right) +\Upsilon_{k},
\]
where $\mu:=\max\left\{ L+\frac{1}{\underline\eta},  L c_{2}+\frac{1}{\underline\eta}c_{1}+V_2, V_2\right\}$. 
 
\end{lemma}

\begin{proof}

Let  $n=\xi^{k}$ at the $k$-th iteration.  It follows from \eqref{Anrk_update} that
            \begin{eqnarray}
		  \begin{aligned}
			A_{n}^{k+1} &= \underset{A_{n,r}}{\arg\min}\,\,\sum_{r=1}^R h_{n}\left(A_{n,r}\right)+\sum_{r=1}^R\langle\tilde{\nabla}_{A_{n,r}}f(\underline{A}_n^k), A_{n,r}-(\tilde{A}_{n,r})^{k}\rangle+\frac{1}{2\eta^{k}} \sum_{r=1}^R\|A_{n,r}-(\tilde{A}_{n,r})^{k}\|_{F}^{2} \\
			 &=\underset{A_{n,r}}{\arg\min}\,\, h_{n}\left(A_{n}\right)+\langle \tilde{\nabla}_{A_{n}}f(\underline{A}_n^k), A_{n}-\tilde{A}_{n}^{k}\rangle+\frac{1}{2\eta^{k}} \|A_{n}-\tilde{A}_{n}^{k}\|_{F}^{2}.
		\end{aligned}  
		\end{eqnarray}
Therefore, we derive that 
\[
0 \in \partial h_{n}(A_{n}^{k+1}) + \tilde{\nabla}_{A_{n}} f(\underline A_n^k) + \frac{1}{\eta^k}( A_{n}^{k+1} - \tilde{A}_{n}^k ).
\]
Combining this with  
$$\partial_n \Phi\left(A^{k+1}\right)=\nabla_{A_n}f\left(A^{k+1}\right)+\partial h_n(A_{n}^{k+1}),
$$ 
we can assert that $P_{n}^{k+1}\in \partial_n \Phi(A^{k+1})$.

The next step is to derive a bound for the expected value of the norm of $P^{k+1}$. Let $n=\xi^{k}$ denote the index selected at the $k$-th iteration, which suggests that
\begin{align*}
 &{\mathbb{E}_{k}\|P^{k+1}\|_{F}}\\
=&{\mathbb{E}_{k}\|P_{\xi^{k}}^{k+1}\|_{F}}\\
=&\mathbb{E}_{k}\|\nabla_{A_{\xi^{k}}}f(A^{k+1}-\tilde{\nabla}_{A_{\xi^{k}}}f(\underline{A}^k))+\frac{1}{\eta^k}(\tilde{A}_{\xi^{k}}^{k}-A_{\xi^{k}}^{k+1})\|_{F}\\
\le&\mathbb{E}_{k}\|\nabla_{A_{\xi^{k}}}f(A^{k+1})-\tilde{\nabla}_{A_{\xi^{k}}}f(\underline{A}^k)\|_{F}+\frac{1}{\eta^k}\mathbb{E}_{k}\|\tilde{A}_{\xi^{k}}^{k}-A_{\xi^{k}}^{k+1}\|_{F}\\
\le&\mathbb{E}_{k}\|\nabla_{A_{\xi^{k}}}f(A^{k+1})-\nabla_{A_{\xi^{k}}}f(\underline{A}^k)\|_{F}+\mathbb{E}_{k}\| \nabla_{A_{\xi^{k}}}f(\underline{A}^k)-\tilde{\nabla}_{A_{\xi^{k}}}f(\underline{A}^k)\|_{F}+\frac{1}{\eta^k}\mathbb{E}_{k}\|\tilde{A}_{\xi^{k}}^{k}-A_{\xi^{k}}^{k+1}\|_{F}\\
\le& L \mathbb{E}_{k}\|A^{k+1}-\underline{A}^k\|_{F}+ \Upsilon_{k}+ V_2 \,(\sum_{i=1}^{t+1}\left\|A^{k+1-i}-A^{k-i}\right\|_{F})+\frac{1}{\eta^k}\mathbb{E}_{k}\|\tilde{A}^{k}-A^{k+1}\|_{F}\\
\le&  L \mathbb{E}_{k}\|A^{k+1}-{A}^k\|_{F}+  L \|\sum_{i=1}^t \beta^{k+1-i}(A^{k+1-i}-A^{k-i})\|_{F}+ \Upsilon_{k}+ V_2\sum_{i=1}^{t}\left\|A^{k+1-i}-A^{k-i}\right\|_{F}\\
&+ V_2\left\|A^{k-t}-A^{k-t-1}\right\|_{F}+\frac{1}{\eta^k}\mathbb{E}_{k}\|A^{k+1}-{A}^{k}\|_{F}+\frac{1}{\eta^k}\|\sum_{i=1}^t \alpha^{k+1-i}(A^{k+1-i}-A^{k-i})\|_{F}\\
\le& ( L+\frac{1}{\eta^k})\mathbb{E}_{k}\|A^{k+1}-{A}^{k}\|_{F} + \sum_{i=1}^t \left( L \beta^{k+1-i}+\frac{1}{\eta^k}\alpha^{k+1-i}+V_2\right)\|A^{k+1-i}-A^{k-i}\|_{F}\\
&+ V_2\left\|A^{k-t}-A^{k-t-1}\right\|_{F}+ \Upsilon_{k}\\
\le&( L+\frac{1}{\underline\eta})\mathbb{E}_{k}\|A^{k+1}-{A}^{k}\|_{F} + \left( L \beta^{k}+\frac{1}{\underline\eta}\alpha^{k}+V_2\right)\sum_{i=1}^t \|A^{k+1-i}-A^{k-i}\|_{F}\\
&+ V_2\left\|A^{k-t}-A^{k-t-1}\right\|_{F}+ \Upsilon_{k}\\
\le&( L+\frac{1}{\underline\eta})\mathbb{E}_{k}\|A^{k+1}-{A}^{k}\|_{F} + \left( L c_{2}+\frac{1}{\underline\eta}c_{1}+V_2\right)\sum_{i=1}^t \|A^{k+1-i}-A^{k-i}\|_{F}\\
&+ V_2\left\|A^{k-t}-A^{k-t-1}\right\|_{F}+ \Upsilon_{k},
\end{align*}
where the last inequality is derived from the sequences 
$\{\alpha^{k+1-i}\}_{i=1}^t$ and $\{\beta^{k+1-i}\}_{i=1}^t$ are both non-decreasing, i.e., 
$\alpha^{k}\ge \alpha^{k-1}$ and $\beta^{k}\ge \beta^{k-1}$, with limits $\lim_{k \rightarrow \infty}\alpha_{k}=c_{1}$ and $\lim_{k \rightarrow \infty}\beta_{k}=c_{2}$.

Let $\mu:=\max\left\{ L+\frac{1}{\underline\eta},  L c_{2}+\frac{1}{\underline\eta}c_{1}+V_2, V_2\right\}$, we have
\[
\mathbb{E}_{k}\|P^{k+1}\|_{F} \le \mu(\mathbb{E}_{k}\|A^{k+1}-A^k\|_{F}+\sum_{i=1}^t \|A^{k+1-i}-A^{k-i}\|_{F}+\|A^{k-t}-A^{k-t-1}\|_{F}) +\Upsilon_{k},
\]
This proves the statement.
\end{proof}

\begin{lemma}\label{lemma_dist2}
Under the same conditions outlined in Lemma \ref{subgradient_bound}, there exists a positive constant 
$\bar\mu:=3\max\left\{6L^2+\frac{6}{\underline\eta}, 6L^2 t\, c_2^2+3V_1+\frac{6t}{\underline\eta}c_1^2, 3V_1\right\}$ such that
\begin{eqnarray*}
\begin{aligned}
\mathbb{E}[\mathrm{dist}(0,\partial \Phi(A^{k+1}))^{2}]&\le \bar{\mu}\left(\mathbb{E}[\|A^{k+1}-A^{k}\|_{F}^{2}]+\sum_{i=1}^{t+1}\|A^{k+1-i}-A^{k-i}\|_{F}^2\right)+3\mathbb{E}\Gamma_{k}.
\end{aligned}
\end{eqnarray*}
\end{lemma}

\begin{proof}
It follows from Lemma \ref{subgradient_bound}  that
\begin{align*}
&\mathbb{E}_{k}\|P^{k+1}\|_{F}^{2}\\
=&\mathbb{E}_{k}\|P_{\xi^{k}}^{k+1}\|_{F}^{2}\\
=&\mathbb{E}_{k}\|\nabla_{A_{\xi^{k}}}f(A^{k+1})-{\nabla}_{A_{\xi^{k}}}f(\underline{A}^k)+\nabla_{A_{\xi^{k}}}f(\underline{A}^k)-\tilde{\nabla}_{A_{\xi^{k}}}f(\underline{A}^k)+\frac{1}{\eta^k}(\tilde{A}_{\xi^{k}}^{k}-A_{\xi^{k}}^{k+1})\|_{F}^{2}\\
\le&3\mathbb{E}_{k}\|\nabla_{A_{\xi^{k}}}f(A^{k+1})-{\nabla}_{A_{\xi^{k}}}f(\underline{A}^k)\|_{F}^{2}+3\mathbb{E}_{k}\|\nabla_{A_{\xi^{k}}}f(\underline{A}^k)-\tilde{\nabla}_{A_{\xi^{k}}}f(\underline{A}^k)\|_{F}^{2}+\frac{3}{\eta^k}\mathbb{E}_{k}\|\tilde{A}^{k}-A^{k+1}\|_{F}^{2}\\
\le&3L^2\mathbb{E}_{k}\|A^{k+1}-\underline{A}^k\|_{F}^{2}+3\Gamma_{k}+3V_1\sum_{i=1}^{t+1} \left\|A^{k+1-i}-A^{k-i}\right\|_{F}^2+\frac{3}{\eta^k}\mathbb{E}_{k}\|A^{k+1}-\tilde{A}^{k}\|_{F}^{2}\\
\le& 6L^2\mathbb{E}_{k}\|A^{k+1}-{A}^k\|_{F}^{2}+6L^2 t\sum_{i=1}^t \left(\beta^{k+1-i}\right)^2\|A^{k+1-i}-A^{k-i}\|_{F}^2+3\Gamma_{k}\\
&+3V_1\sum_{i=1}^{t+1} \left\|A^{k+1-i}-A^{k-i}\right\|_{F}^2+\frac{6}{\eta^k}\mathbb{E}_{k}\|A^{k+1}-{A}^{k}\|_{F}^{2}+\frac{6t}{\eta^k}\sum_{i=1}^t \left(\alpha^{k+1-i}\right)^2\left\|A^{k+1-i}-A^{k-i}\right\|_{F}^2\\
\le& (6L^2+\frac{6}{\underline\eta})\,\mathbb{E}_{k}\|A^{k+1}-{A}^k\|_{F}^{2}+\left(6L^2 t\left(\beta^{k}\right)^2+3V_1+\frac{6t}{\underline\eta}\left(\alpha^{k}\right)^2\right)\sum_{i=1}^t \|A^{k+1-i}-A^{k-i}\|_{F}^2\\
&+3V_1 \left\|A^{k-t}-A^{k-t-1}\right\|_{F}^2+3\Gamma_{k}\\
\le& (6L^2+\frac{6}{\underline\eta})\,\mathbb{E}_{k}\|A^{k+1}-{A}^k\|_{F}^{2}+\left(6L^2 t\, c_2^2+3V_1+\frac{6t}{\underline\eta}c_1^2\right)\sum_{i=1}^t \|A^{k+1-i}-A^{k-i}\|_{F}^2\\
&+3V_1 \left\|A^{k-t}-A^{k-t-1}\right\|_{F}^2+3\Gamma_{k}\\
\le&\bar\mu(\mathbb{E}_{k}\|A^{k+1}-{A}^k\|_{F}^{2}+\sum_{i=1}^t\|A^{k+1-i}-A^{k-i}\|_{F}^2+\left\|A^{k-t}-A^{k-t-1}\right\|_{F}^2)+3\Gamma_{k},
\end{align*}
where $\bar\mu^0:=\max\left\{6L^2+\frac{6}{\underline\eta}, 6L^2 t\, c_2^2+3V_1+\frac{6t}{\underline\eta}c_1^2, 3V_1\right\}$. 

Utilizing $\mathbb E_k \left[\mathrm{dist}\left(0, \partial \Phi\left(A^{k+1}\right)\right)^2 \right]\leq 3\mathbb E_k\left [\left\|P^{k+1}\right\|_{F}^2\right]$ and taking the full expectation on both sides, we derive the following result
\begin{eqnarray*}
\begin{aligned}
&\mathbb{E}[\mathrm{dist}(0,\partial \Phi(A^{k+1}))^{2}]\\
\le &\bar{\mu}\left(\mathbb{E}\|A^{k+1}-A^{k}\|_{F}^{2}+\sum_{i=1}^t\|A^{k+1-i}-A^{k-i}\|_{F}^2+\left\|A^{k-t}-A^{k-t-1}\right\|_{F}^2\right)+3\mathbb{E}\Gamma_{k},
\end{aligned}
\end{eqnarray*}
where $\bar\mu=3\bar\mu^0$. 
Hence, the proof is finalized.
\end{proof}

\begin{theorem}\label{subgradient_rate}
Suppose Assumptions \ref{assume_ipsgd} holds, the stepsize $\eta^k$ lies in $[\underline \eta, \bar\eta]$ and \eqref{paraofLya} is satisfied.  Let $\{A_1^{k}, A_2^{k}, A_3^{k}\}_{k\in\mathbb{N}}$ produced by Algorithm \ref{ipSGD} remain bounded for all $k$. Then there exists $0<\sigma<\epsilon$ such that
\[
\mathbb{E}[\mathrm{dist}(0,\partial \Phi(A^{\hat{k}}))^{2}]\le \frac{\bar{\mu}}{(\epsilon-\sigma)K}\left(\mathbb{E}\Psi_{1}+\frac{3(\epsilon-\sigma)}{\tau\bar{\mu}}\mathbb{E}\Gamma_{1}\right)=\mathcal{O}(1/K),
\]
where $\hat{k}$ is drawn from $\{2, \dots, K+1\}$. This implies that at most  $\mathcal{O}(\epsilon^{-2})$ iterations are required in expectation to reach an $\epsilon$-stationary point (refer to Definition \ref{stationary-point}) of $\Phi$. 
\end{theorem}

\begin{proof}
Combining \eqref{decent_inequality_01} and Lemma \ref{lemma_dist2}, it shows that

\begin{align*}
&\mathbb{E}[\Psi_{k}-\Psi_{k+1}]\\
\ge&\epsilon(\,\mathbb{E}\|A^{k+1}-A^{k}\|_{F}^{2}+\sum_{i=1}^{t+1}\mathbb{E}\|A^{k+1-i}-A^{k-i}\|_{F}^2)\\
\ge&\sigma(\,\mathbb{E}\|A^{k+1}-A^{k}\|_{F}^{2}+\sum_{i=1}^{t+1}\mathbb{E}\|A^{k+1-i}-A^{k-i}\|_{F}^2)+\frac{\epsilon-\sigma}{\bar{\mu}}\mathbb{E}[\mathrm{dist}(0,\partial \Phi(A^{k+1}))^{2}]-\frac{3(\epsilon-\sigma)}{\bar{\mu}}\mathbb{E}\Gamma_{k}\\
\ge&\sigma(\,\mathbb{E}\|A^{k+1}-A^{k}\|_{F}^{2}+\sum_{i=1}^{t+1}\mathbb{E}\|A^{k+1-i}-A^{k-i}\|_{F}^2)+\frac{\epsilon-\sigma}{\bar{\mu}}\mathbb{E}[\mathrm{dist}(0,\partial \Phi(A^{k+1}))^{2}]\\
&+\frac{3(\epsilon-\sigma)}{\tau\bar{\mu}}\mathbb{E}[\Gamma_{k+1}-\Gamma_{k}]-\frac{3(\epsilon-\sigma)V_{\Gamma}}{\tau\bar{\mu}}\sum_{i=1}^{t+1}\mathbb{E}\|A^{k+1-i}-A^{k-i}\|_{F}^2\\
\ge&\sigma(\,\mathbb{E}\|A^{k+1}-A^{k}\|_{F}^{2}+\sum_{i=1}^{t+1}\mathbb{E}\|A^{k+1-i}-A^{k-i}\|_{F}^2)+\frac{\epsilon-\sigma}{\bar{\mu}}\mathbb{E}[\mathrm{dist}(0,\partial \Phi(A^{k+1}))^{2}]\\
&+\frac{3(\epsilon-\sigma)}{\tau\bar{\mu}}\mathbb{E}[\Gamma_{k+1}-\Gamma_{k}]-\frac{3(\epsilon-\sigma)V_{\Gamma}}{\tau\bar{\mu}}\sum_{i=1}^{t+1}\mathbb{E}\|A^{k+1-i}-A^{k-i}\|_{F}^2-\frac{3(\epsilon-\sigma)V_{\Gamma}}{\tau\bar{\mu}}\mathbb{E}\|A^{k+1}-A^{k}\|_{F}^{2},
\end{align*}
where the third inequality is derived from \eqref{Gamma_k1_k} in Definition \ref{vr_definition}. Let $\sigma=\frac{3(\epsilon-\sigma)V_{\Gamma}}{\tau\bar{\mu}}$, i.e., $\sigma=\frac{3\epsilon  V_{\Gamma}}{\tau\bar{\mu}+3V_{\Gamma}}$, it shows that
	\[
	\mathbb{E}[\Psi_{k}-\Psi_{k+1}]\ge \frac{\epsilon-\sigma}{\bar{\mu}}\mathbb{E}[\mathrm{dist}(0,\partial \Phi(A^{k+1}))^{2}]+\frac{3(\epsilon-\sigma)}{\tau\bar{\mu}}\mathbb{E}[\Gamma_{k+1}-\Gamma_{k}].
	\]
	Summing up from $k=1$ to $K$, we can conclude that
	\[
	\mathbb{E}[\Psi_{1}-\Psi_{K+1}]\ge \frac{\epsilon-\sigma}{\bar{\mu}}\sum_{k=1}^{K}\mathbb{E}[\mathrm{dist}(0,\partial \Phi(A^{k+1}))^{2}]+\frac{3(\epsilon-\sigma)}{\tau\bar{\mu}}\mathbb{E}[\Gamma_{K+1}-\Gamma_{1}],
	\]
      implying the existence of $\hat{k} \in{2,\dots,K+1}$ such that
	\begin{align*}
		\mathbb{E}[\mathrm{dist}(0,\partial \Phi(A^{\hat{k}}))^{2}]\le&\frac{1}{K}\sum_{k=1}^{K}\mathbb{E}[\mathrm{dist}(0,\partial \Phi(A^{k+1}))^{2}]\\
		\le&\frac{\bar{\mu}}{(\epsilon-\sigma)K}(\mathbb{E}[\Psi_{1}-\Psi_{K+1}]+\frac{3(\epsilon-\sigma)}{\tau\bar{\mu}}\mathbb{E}[\Gamma_{1}-\Gamma_{K+1}])\\
		\le&\frac{\bar{\mu}}{(\epsilon-\sigma)K}(\mathbb{E}\Psi_{1}+\frac{3(\epsilon-\sigma)}{\tau\bar{\mu}}\mathbb{E}\Gamma_{1}).
	\end{align*}
The proof is completed.
\end{proof}
Define the set of limit points of $\{A^{k}\}_{k\in\mathbb{N}}$ as
\begin{eqnarray*}
	\begin{aligned}
		\Omega(A^{0}):=\{A^{*}:\exists &\text{ an increasing sequence of integers }\{k_{l}\}_{l\in\mathbb{N}} \text{ such that }A^{k_{l}}\rightarrow A^{*} \text{ as } l\rightarrow+\infty \}.
	\end{aligned}
\end{eqnarray*}

Next, by combining Theorems~\ref{subsequence_convergence} and \ref{subgradient_rate} with the K{\L} property \cite{BolteST14, DriggsTLDS2020}, we are ready to outline the global convergence result for the Midas-LL1 algorithm in the following theorem. The proof details are omitted here for brevity, but readers may refer to \cite{DriggsTLDS2020, WangH23}.

\begin{theorem}\label{global_convergence}
Suppose that Assumption \ref{assume_ipsgd} is satisfied, the step $\eta_{k}$ is within the range $[\underline \eta, \bar\eta]$ and \eqref{paraofLya} is satisfied. Let the sequence  $\{A^{k}\}_{k\in\mathbb{N}}$ generated by Algorithm \ref{ipSGD} with variance-reduced gradient estimator be bounded, where $A^{k}=\{A_1^{k}, A_2^{k}, A_3^{k}\}$. $\Phi$ is a semialgebraic function that satisﬁes the K{\L} property with exponent $\theta \in [0, 1)$, then either the point $A^{k}$ reaches a critical point after a ﬁnite number of iterations or the sequence $\{A^{k}\}_{k\in\mathbb{N}}$ almost surely demonstrates the ﬁnite length property in expectation,

\[
\sum_{k=0}^{+\infty}\mathbb{E}\|A^{k+1}-A^{k}\|_{F}<+\infty.
\]
\end{theorem}
\section{Numerical experiments}\label{numercial_experiments}

In this section, we present the effectiveness of the proposed Midas-LL1 (Algorithm \ref{ipSGD}) through numerical experiments conducted on two real-world hyperspectral sub-image (HSI) datasets \footnote{\url{http://www.ehu.eus/ccwintco/index.php/Hyperspectral_Remote_Sensing_Scenes.}}, Salinas and Pavia, and two video datasets, Carphone and Foreman.  We aim to demonstrate its superior efficiency through comparisons with the state-of-the-art algorithm ALS-MU-based LL1 algorithms, namely, MVNTF in \cite{qian2017}, which also does not have any structural regularization on the factors except for the nonnegativity constraints.

To evaluate its performance, Midas-LL1 is implemented with three stochastic gradient estimators: vanilla SGD, the unbiased variance-reduced SAGA \cite{DefazioBL14}, and the biased variance-reduced SARAH \cite{NguyenLST17}. This study facilitates a comparative analysis of gradient bias properties and variance reduction strategies. Additionally, we consider multi-step inertial acceleration with $t=3$, one-step inertial acceleration with $t=1$, and also no  acceleration version  with $t=0$, respectively. The combinations of these stochastic gradient estimators and multi-step inertial acceleration schemes are denoted as 3-Midas-SGD/SAGA/SARAH, 1-Midas-SGD/SAGA/SARAH, and 0-Midas-SGD/SAGA/SARAH, respectively.


Because the proposed algorithms under test involve different operations and subproblem-solving strategies, finding a unified complexity measure can be challenging, especially with multi-step acceleration that varies with different $t$. To address this issue, 
all algorithms are run for a fixed 200 epochs, and we present the final outcomes obtained. When handling
large-scale tensor decomposition problems in video dataset experiments, the MVNTF algorithms typically
run with extra lengthy time, yet fail to meet the specified stopping criterion. In this case,  
we set the maximal number of iterations to be 1000. 
Additionally, we evaluate the cost value of the algorithms about runtime, providing a comprehensive comparison of performance.

In all experiments,  we set $L_r=L$, for $r=1,\ldots,R$, stepsize $\eta^k=\eta=0.1$,   the batchsize of sampled fibers  $B=2L$,  {the inertial parameters $\alpha^k=0.3\times\frac{k-1}{k+2}$ and  $\beta^{k}=0.8\times\frac{k-1}{k+2}$.}

The peak signal-to-noise ratio (PSNR) values are considered to measure the numerical performance for the real dataset experiments. The PSNR value is defined as
\[
\text{PSNR}=10\log_{10}\left(\frac{\mathcal{X}^2_{\max}\prod_{i=1}^{N}I_{i}}{\|\bar{\mathcal{X}}-\mathcal{X}\|_F^2}\right),
\]
where $\mathcal{X}_{\max}$ is the maximum intensity of the original dataset $\mathcal{X}$ and $\bar{\mathcal{X}}=\sum_{r=1}^R(A_{1,r} \cdot A_{2,r}^{\top}) \circ \boldsymbol{c}_r$ denotes the rank-$\left(L, L, 1\right)$ BTD approximation of $\mathcal{X}$.   We also
adopt cross correlation (CC), root mean square error
(RMSE), spectral angle mapper (SAM) to illustrate the stability of all algorithms.

\subsection{Comparisons of zero-step, one-step, and three-step Midas-LL1}

In this subsection, we present the performance of various methods on hyperspectral sub-image datasets (dimensions: height × width × bands) to compare the effectiveness of Midas-LL1 with zero-step, one-step, and three-step accelerations. Hyperspectral images (HSI) capture both spatial and spectral information by recording reflectance values across multiple wavelengths, making them a valuable tool for applications such as remote sensing and material classification. The spectral–spatial joint structure of HSI is naturally suited for Midas-LL1 algorithms because rank-$\left(L_r, L_r, 1\right)$ block-term decomposition provides a particularly effective framework for modeling this structure. The datasets used for the experiments include Salinas (\(80 \times 84 \times 204\)), which focuses on agricultural regions, and Pavia Centre (\(288 \times 352 \times 300\)), which captures an urban environment with intricate spatial patterns.

\subsubsection{Salinas}
We first present the numerical experiments on  Salinas dataset with different values of $R$ and $L$  in Figure \ref{Salinas}. Midas-LL1 with multi-step acceleration (\(t=3\)) outperforms both one-step (\(t=1\)) and zero-step (\(t=0\)) methods. 
{This confirms the effectiveness of our proposed multi-step inertial accelerated   method in solving the tensor BTD problem.}
Notably, Midas-LL1 with the variance-reduced SGD estimator outperforms Midas-LL1 with the vanilla SGD estimator in three-step, one-step, and zero-step settings, with the latter two showing even more pronounced improvements. 


{We continue to show the final performance of the return tensor  in terms of  RMSE, SAM,  CC,  and PSNR} in Table \ref{Salinas_table}. From the results, we see that both  3-Midas-SAGA and 3-Midas-SARAH  achieve superior performance, demonstrating that variance-reduced stochastic gradient estimators effectively enhance reconstruction accuracy. Their consistent performance across different configurations of $R$ and $L$ further highlights the stability and adaptability of these methods for hyperspectral image data.

The rank-\((L_r, L_r, 1)\) block-term tensor decomposition approximation results for the Salinas dataset are shown in Figures  \ref{salinas_compare}. These results demonstrate that the proposed Midas-LL1 method with multi-step acceleration (\(t=3\)) surpasses both one-step (\(t=1\)) and zero-step (\(t=0\)) methods in recovering the original Salinas dataset, as illustrated in Figure \ref{salinas_compare} with varying \(R\) and \(L\). Furthermore, methods employing SAGA and SARAH stochastic gradient estimators consistently achieve better performance compared to those based on SGD, particularly in producing clearer boundaries between different regions in the images.

\begin{figure}[!htb]
	\setlength\tabcolsep{2pt}
	\centering
	\begin{tabular}{ccc}
		\includegraphics[width=0.32\textwidth]{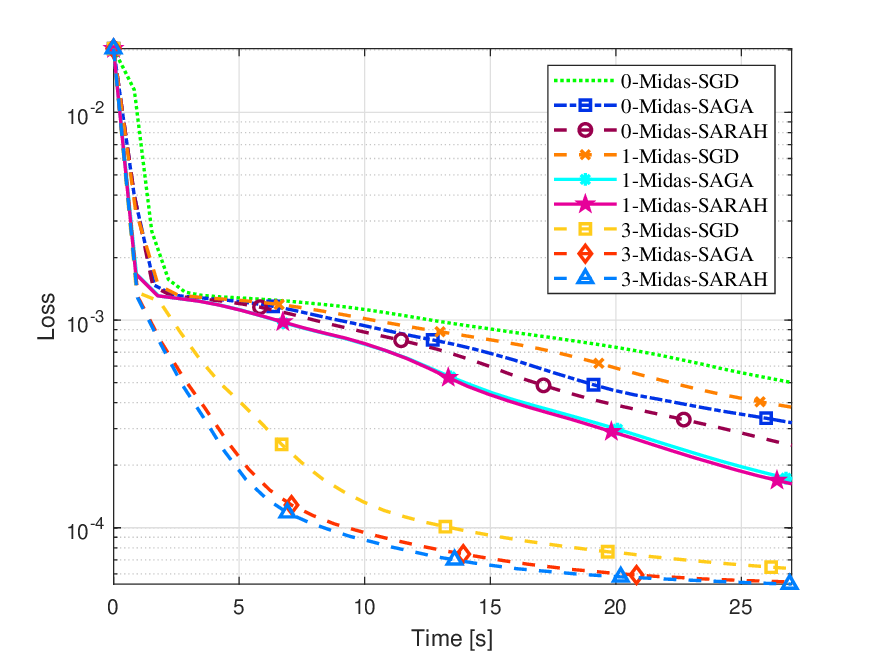}&
		\includegraphics[width=0.32\textwidth]{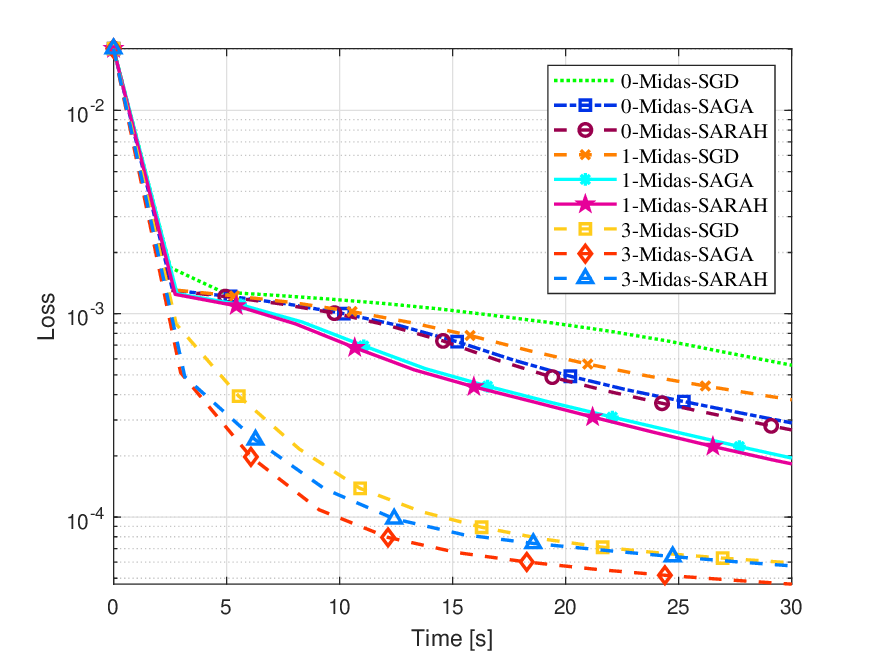}&
        \includegraphics[width=0.32\textwidth]{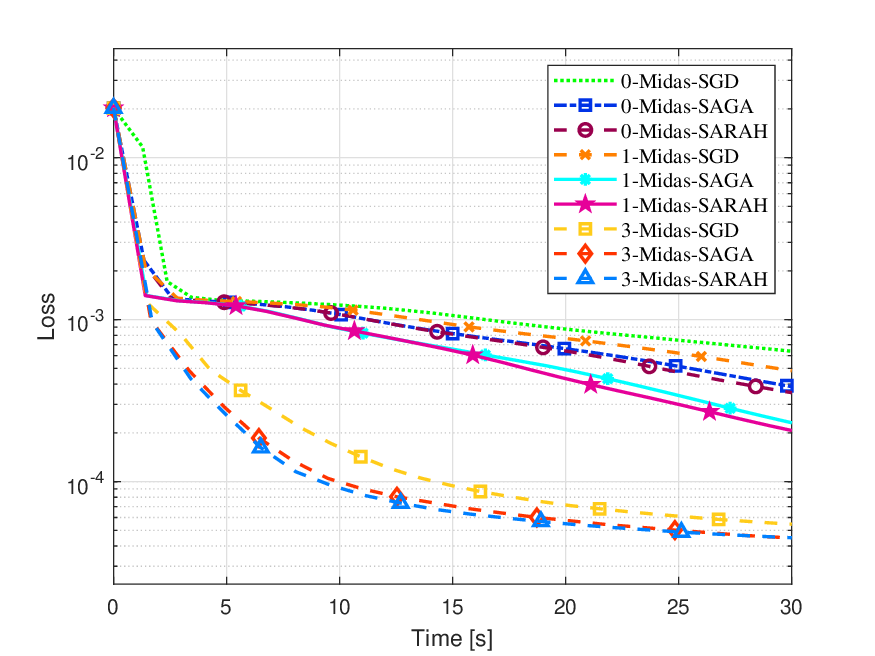}\\
	  (a) $R=2$, $L=15$&(b) $R=3$, $L=12$&(c) $R=4$, $L=10$\\
	\end{tabular}
	\caption{Numerical experiments on  the $80 \times 84 \times 204 $ Salinas dataset with different values of $R$ and $L$. }
	\label{Salinas}
\end{figure}

\begin{table*}[!htb]	
	\centering
	\fontsize{7}{9}\selectfont
	\caption{
		Comparisons of Midas-LL1 with zero-step, one-step, and three-step inertial acceleration on Salinas dataset. The parameters $\alpha^k=0.3\times\frac{k-1}{k+2}$, $\beta^{k}=0.8\times\frac{k-1}{k+2}$, and $\eta^{k}=0.1$.
		``$a_0$'' denotes 0-Midas-SGD; ``$a_1$'' denotes 0-Midas-SAGA; ``$a_2$'' denotes 0-Midas-SARAH; ``$b_0$'' denotes 1-Midas-SGD; ``$b_1$'' denotes 1-Midas-SAGA; ``$b_2$'' denotes 1-Midas-SARAH;  ``$c_0$'' denotes 3-Midas-SGD; ``$c_1$'' denotes 3-Midas-SAGA; ``$c_2$'' denotes 3-Midas-SARAH.}
	\label{Salinas_table}
	\begin{footnotesize}
	  
	\begin{tabular}{c|c|c| c c c| c c c |c c c}
		\hline
		$R$&$L$&Index &$a_0$&$a_1$&$a_2$ &$b_0$&$b_1$&$b_2$ &$c_0$&$c_1$&$c_2$ 
		\cr\hline	
		\multirow{4}{*}{2}&\multirow{4}{*}{15}&RMSE (0)&0.0413&0.0211&0.0210&0.0247 &0.0151&0.0150&0.0109&0.0101&\textbf{0.0101}\\
		&&SAM (0)&0.1585&0.0718&0.0716&0.0864 &0.0524&0.0523&0.0406&0.0382&\textbf{0.0379}\\

		&&CC (1)&0.1153&0.8178&0.8191&0.7328&0.8977&0.8985&0.9269&0.9346&\textbf{0.9354}\\
		&&PSNR ($\infty$)&27.676&33.528&33.543&32.147&36.440&36.483&39.231&39.934&\textbf{39.946}\\ \cline{1-12}
		\multirow{4}{*}{3}&\multirow{4}{*}{12}&RMSE (0)&0.0281&0.0197&0.0195&0.0225 &0.0152&0.0153&0.0101&\textbf{0.0087}&0.0093\\
		&&SAM (0)&0.0969&0.0676&0.0669&0.0767&0.0527&0.0532&0.0392&\textbf{0.0346}&0.0359\\
		&&CC (1)&0.6063&0.8235&0.8273&0.7676&0.8851&0.8843&0.9309&\textbf{0.9438}&0.9433\\
		&&PSNR ($\infty$)&31.021&34.103&34.207&36.360&38.814&36.333&39.886&\textbf{41.218}&40.644\\\cline{1-12}
		\multirow{4}{*}{4}&\multirow{4}{*}{10}&RMSE (0)&0.0277&0.0175&0.0175&0.0193 &0.0133&0.0132&0.0090&0.0083&\textbf{0.0082}\\
		&&SAM (0)&0.1002&0.0620&0.0626&0.0685&0.0479&0.0479&0.0350&\textbf{0.0326}&0.0331\\
		&&CC (1)&0.5754&0.8625&0.8628&0.8399&0.9071&0.9072&0.9446&\textbf{0.9513}&0.9500\\
		&&PSNR ($\infty$)&31.144&35.145&35.115&34.283&37.549&37.579&40.877&41.606&\textbf{41.715}\\\hline
		
	\end{tabular}
	\end{footnotesize}
\end{table*}

\begin{figure}[!htb]
	
	\centering
    \includegraphics[width=1\linewidth]{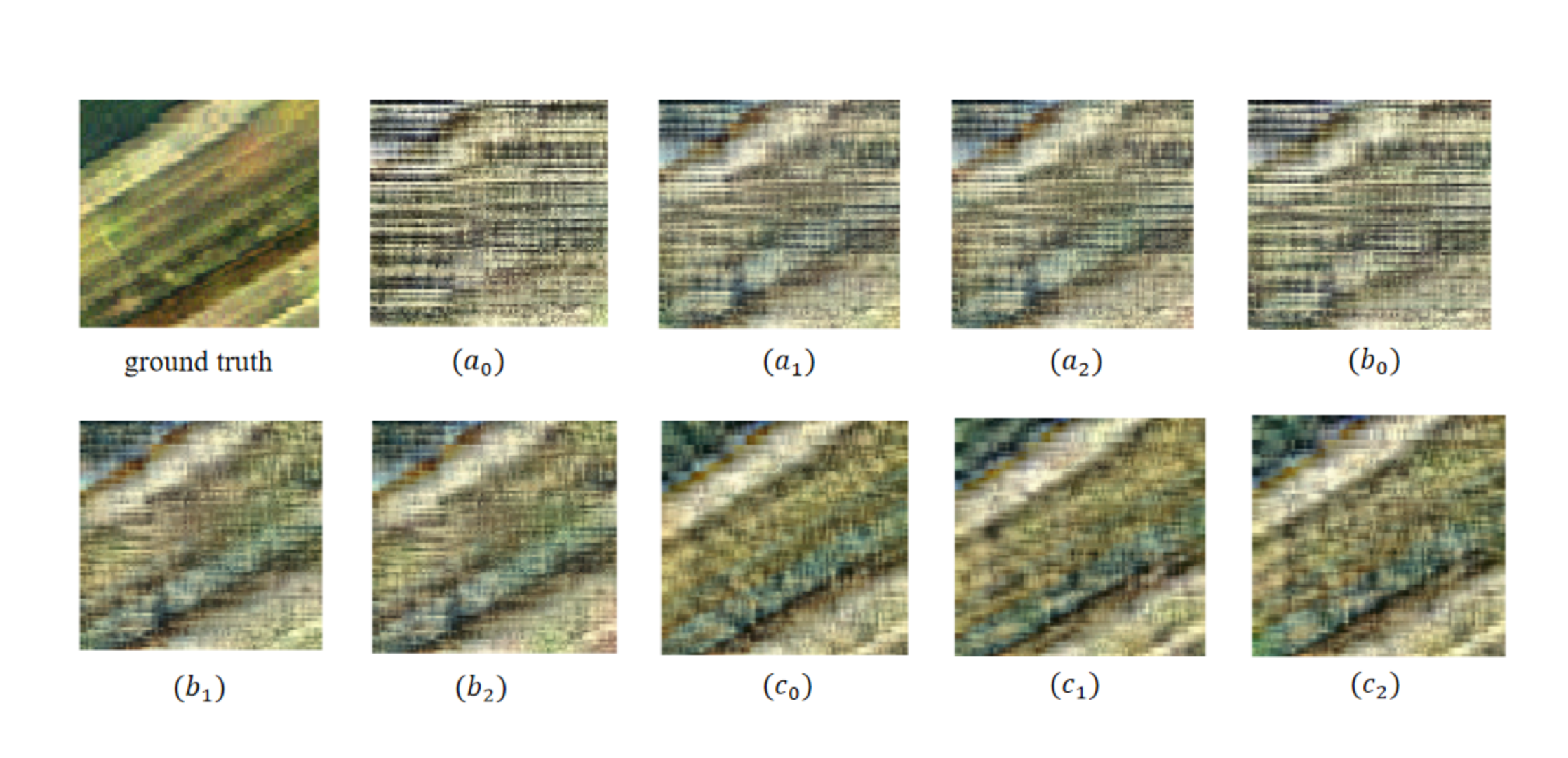}
	
	\caption{The results of Midas-LL1 for rank-\((L_r, L_r, 1)\) block-term tensor decomposition approximation on the 15-th frame of the \emph{Salinas} dataset with \(R=3\) and \(L=12\). 
		``$a_0$'' denotes 0-Midas-SGD; ``$a_1$'' denotes 0-Midas-SAGA; ``$a_2$'' denotes 0-Midas-SARAH; ``$b_0$'' denotes 1-Midas-SGD; ``$b_1$'' denotes 1-Midas-SAGA; ``$b_2$'' denotes 1-Midas-SARAH;  ``$c_0$'' denotes 3-Midas-SGD; ``$c_1$'' denotes 3-Midas-SAGA; ``$c_2$'' denotes 3-Midas-SARAH.}
	\label{salinas_compare}
\end{figure}

\subsubsection{Pavia Centre}
Next, we analyze the performance of the algorithms on the Pavia Centre dataset, and the results are presented in Figure \ref{Pavia} and Table \ref{Pavia_table}. 
{Again, compared with  Midas-LL1 Methods with one-step or zero-step acceleration,} Midas-LL1 Methods with three-step acceleration (3-Midas-SGD, 3-Midas-SAGA, 3-Midas-SARAH) demonstrate the fastest convergence and achieve the lowest loss across all configurations of $R$ and $L$, indicating the effectiveness of multi-step acceleration in improving performance. Among the three-step methods, 3-Midas-SAGA and 3-Midas-SARAH consistently outperform 3-Midas-SGD in terms of both convergence speed and final loss, highlighting the advantages of variance reduction techniques.
\begin{figure}[!htb]
	\setlength\tabcolsep{2pt}
	\centering
	\begin{tabular}{ccc}
		\includegraphics[width=0.32\textwidth]{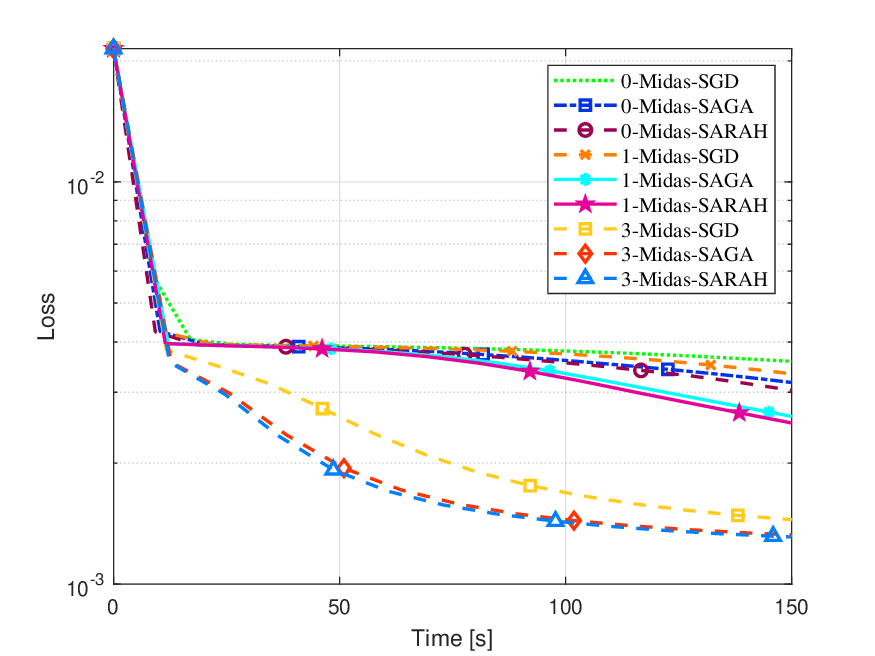}&
		\includegraphics[width=0.32\textwidth]{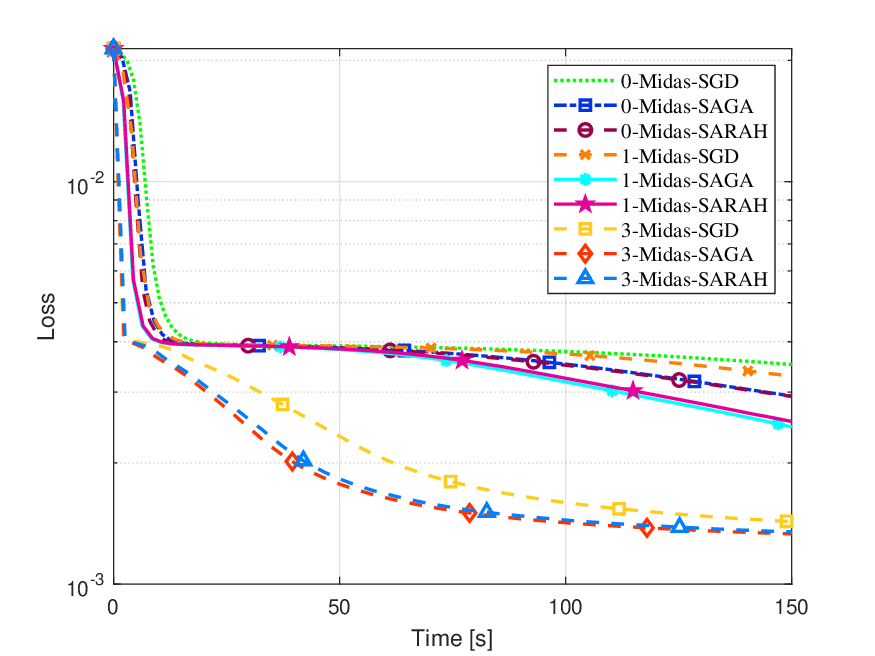}&
        \includegraphics[width=0.32\textwidth]{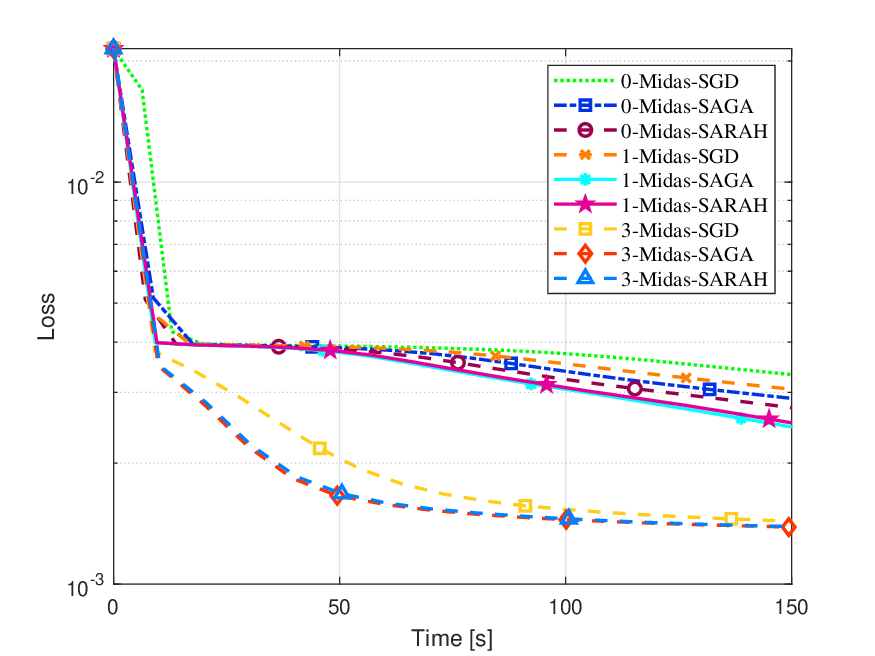}\\
		(a) $R=2$, $L=35$&(b) $R=3$, $L=25$&(c) $R=4$, $L=15$\\
	\end{tabular}
	\caption{Numerical experiments on  the $610 \times 340 \times 103 $ Pavia Centre dataset with different values of $R$ and $L$. }
	\label{Pavia}
\end{figure}

\begin{table*}[!htb]	
	\centering
	\fontsize{7}{9}\selectfont
	\caption{
		Comparisons of Midas-LL1 with zero-step, one-step, and three-step inertial acceleration on Pavia Centre dataset. The parameters $\alpha^k=0.3\times\frac{k-1}{k+2}$, $\beta^{k}=0.8\times\frac{k-1}{k+2}$, and $\eta^{k}=0.1$.
		``$a_0$'' denotes 0-Midas-SGD; ``$a_1$'' denotes 0-Midas-SAGA; ``$a_2$'' denotes 0-Midas-SARAH; 
        ``$b_0$'' denotes 1-Midas-SGD; ``$b_1$'' denotes 1-Midas-SAGA; ``$b_2$'' denotes 1-Midas-SARAH;  ``$c_0$'' denotes 3-Midas-SGD; ``$c_1$'' denotes 3-Midas-SAGA; ``$c_2$'' denotes 3-Midas-SARAH.}
	\label{Pavia_table}
	\begin{footnotesize}
	  
	\begin{tabular}{c|c|c| c c c| c c c |c c c}
		\hline
		$R$&$L$&Index &$a_0$&$a_1$&$a_2$ &$b_0$&$b_1$&$b_2$ &$c_0$&$c_1$&$c_2$ 
		\cr\hline	
		\multirow{4}{*}{2}&\multirow{4}{*}{35}&RMSE (0)&0.0739&0.0610&0.0603&0.0627&0.0540 &0.0539&0.0503&\textbf{0.0491}&0.0491\\
		&&SAM (0)
        &0.2453&0.2041&0.1999&0.2071 &0.1685&0.1684&0.1534&0.1501&\textbf{0.1500}\\

		&&CC (1)&0.5951&0.7431&0.7501&0.7277&0.8058&0.8060&0.8328&0.8401&\textbf{0.8402}\\
		&&PSNR ($\infty$)&22.626&24.294&24.397&24.056&25.358&25.360&25.976&\textbf{26.180}&26.177\\ \cline{1-12}
		\multirow{4}{*}{3}&\multirow{4}{*}{25}&RMSE
        (0)&0.0676&0.0550&0.0552&0.0571&0.0517 &0.0518&0.0503&0.0492&\textbf{0.0495}\\
		&&SAM (0)&0.2265&0.1761&0.1764&0.1843&0.1650&0.1648&0.1596&0.1586&\textbf{0.1586}\\
		&&CC (1)&0.6745&0.7995&0.7985&0.7823&0.8248&0.8246&0.8261&\textbf{0.8349}&0.8323\\
		&&PSNR ($\infty$)&23.401&25.188&25.168&24.871&25.724&25.716&25.970&\textbf{26.158}&26.111\\\cline{1-12}
		\multirow{4}{*}{4}&\multirow{4}{*}{15}&RMSE (0)&0.0611&0.0555&0.0555&0.0561 &0.0528&0.0528&0.0510&\textbf{0.0503}&0.0508\\
		&&SAM (0)&0.1902&0.1744&0.1746&0.1757&0.1635&0.1636&0.1557&\textbf{0.1536}&0.1561\\
		&&CC (1)&0.7409&0.7904&0.7903&0.7867&0.8126&0.8126&0.8258&\textbf{0.8308}&0.8274\\
		&&PSNR ($\infty$)&24.285&25.116&25.109&25.028&25.550&25.547&25.854&\textbf{25.976}&25.891\\\hline
		
	\end{tabular}
	\end{footnotesize}
\end{table*}

\subsection{ 
Comparing Midas-LL1 under one-step and three-step acceleration with MVNTF.}

In this subsection, we evaluate two video datasets\footnote{http://trace.eas.asu.edu/yuv/} (with dimensions: length × width × frames) to compare the performance of Midas-LL1 under one-step and three-step acceleration with MVNTF. The datasets used for testing include Carphone ($144 \times 176 \times 396 $) and Foreman ($288 \times 352 \times 300$). These video datasets, collected and maintained by Arizona State University, are widely used in research. 

\subsubsection{Carphone}
 We evaluate the algorithms on the Carphone dataset, terminating all runs at 80 seconds, and present the results in Figure \ref{carphone}. The performance of the three Midas-LL1 methods with $t=3$, namely 3-Midas-SGD, 3-Midas-SAGA, and 3-Midas-SARAH,  consistently outperforms MVNTF and one-step Midas-LL1 methods from $R=3$ to $R=5$. For $R=3$ and $L=2$, the one-step accelerated methods with $t=1$ 
 also outperform MVNTF. However, as $R$ increases, 1-Midas-SGD only demonstrates a very slight improvement over MVNTF, shown in Figure \ref{carphone} (c), while 1-Midas-SAGA and 1-Midas-SARAH exhibit stable performance, highlighting the effectiveness of variance-reduced stochastic gradient estimators. 
 
{We continue to show the final performance of the return tensor  in terms of  RMSE, SAM,  CC,  and PSNR} in Table \ref{Carphone_table}. The results show that  3-Midas-SARAH demonstrates robustness in reconstructing and preserving both spatial and spectral information across various settings. However, when $R=4$, 3-Midas-SGD achieves superior performance. This result highlights that {the potential advantage  of the vanilla SGD approach.}
 
 
\begin{figure}[!htb]
	\setlength\tabcolsep{2pt}
	\centering
	\begin{tabular}{ccc}
		\includegraphics[width=0.32\textwidth]{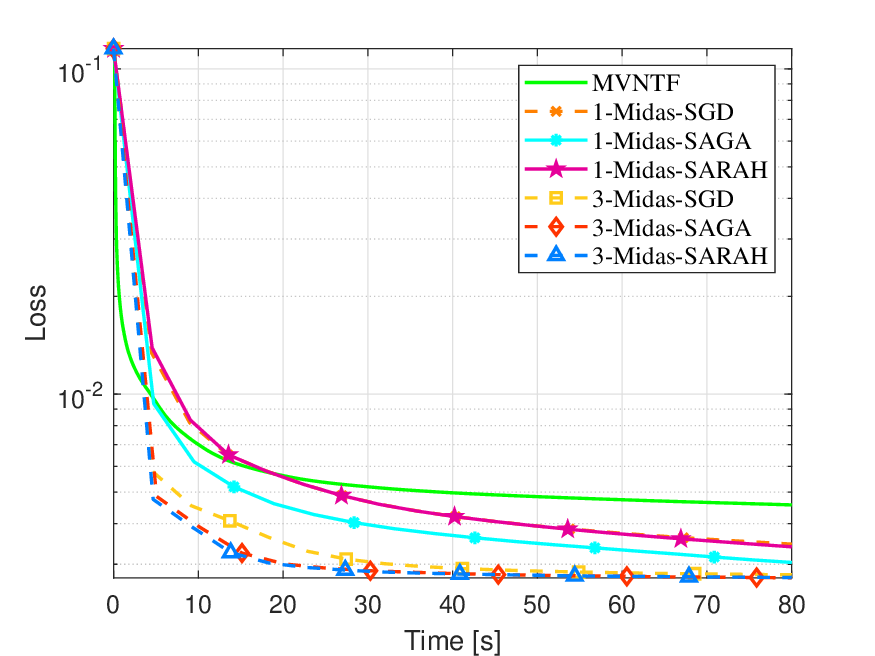}&
		\includegraphics[width=0.32\textwidth]{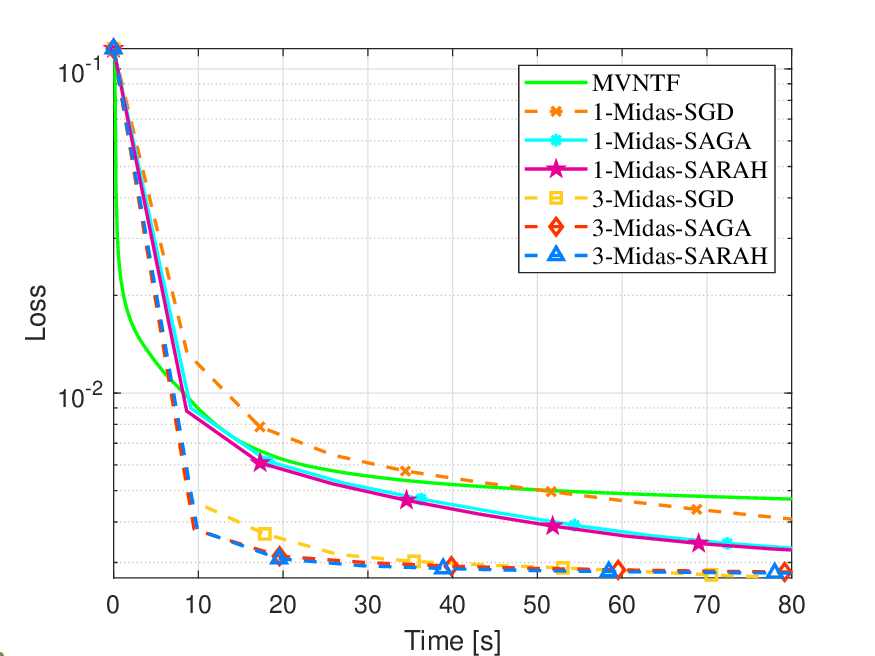}&
        \includegraphics[width=0.32\textwidth]{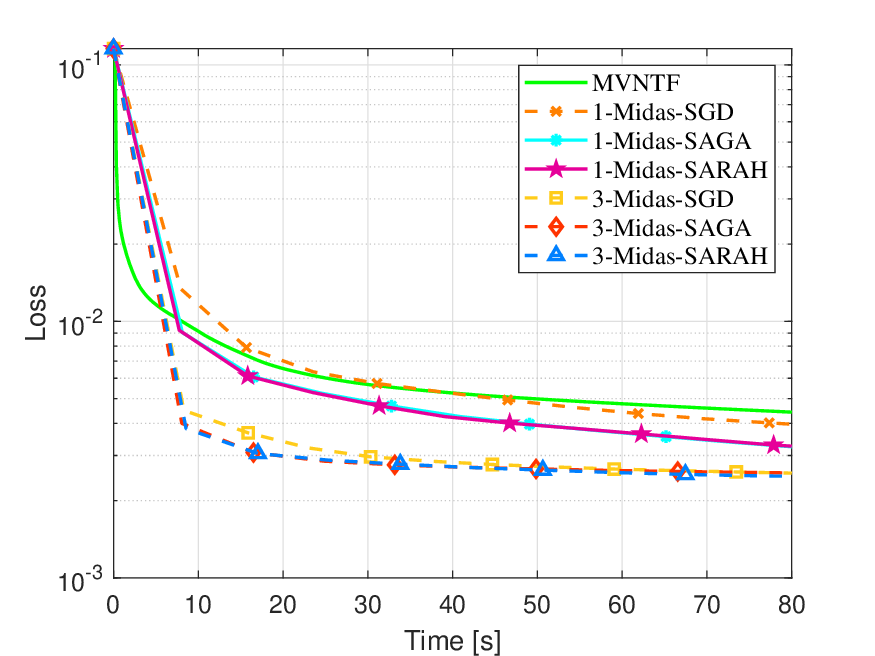}\\
	  (a) $R=3$, $L=20$&(b) $R=4$, $L=14$&(c) $R=5$, $L=10$\\
	\end{tabular}
	\caption{Numerical experiments of Midas-LL1  under 1-step and 3-step acceleration with MVNTF on the $144 \times 176 \times 396 $ Carphone dataset. }
	\label{carphone}
\end{figure}

\begin{table*}[!htb]	
	\centering
	\fontsize{7}{9}\selectfont
	\caption{
		Comparisons of Midas-LL1 under one-step and three-step inertial acceleration with MVNTF on Carphone dataset. The parameters $\alpha^k=0.3\times\frac{k-1}{k+2}$, $\beta^{k}=0.8\times\frac{k-1}{k+2}$, and $\eta^{k}=0.1$.
		``a'' denotes MVNTF; 
        ``$b_0$'' denotes 1-Midas-SGD; ``$b_1$'' denotes 1-Midas-SAGA; ``$b_2$'' denotes 1-Midas-SARAH;  ``$c_0$'' denotes 3-Midas-SGD; ``$c_1$'' denotes 3-Midas-SAGA; ``$c_2$'' denotes 3-Midas-SARAH.}
	\label{Carphone_table}
	\begin{small}
	  
	\begin{tabular}{c|c|c|c| c |c c c |c c c}
		\hline
		Dataset&$R$&$L$&Index & a&
        $b_0$ &  $b_1$ & $b_2$ &
         $c_0$ &  $c_1$ & $c_2$ 
		\cr\hline	
		&\multirow{4}{*}{3}&\multirow{4}{*}{20}&RMSE (0)&0.0877&0.0819&0.0765&0.0750 &0.0744&0.0736&\textbf{0.0732}\\
		&&&SAM (0)&0.2175&0.2071&0.1862&0.1792 &0.1779&0.1751&\textbf{0.1745}\\

		&&&CC (1)&0.9444&0.9514&0.9577&0.9595&0.9600&0.9609&\textbf{0.9613}\\
		&&&PSNR ($\infty$)&21.136&21.735&22.321&22.502&22.566&22.660&\textbf{22.707}\\ \cline{2-11}
		&\multirow{4}{*}{4}&\multirow{4}{*}{14}&RMSE (0)&0.0941&0.0821&0.0776&0.0776 &\textbf{0.0718}&0.0737&0.0739\\
		&&&SAM (0)&0.2391&0.2003&0.1914&0.1916&\textbf{0.1802}&0.1813&0.1810\\
		\emph{Carphone}&&&CC (1)&0.9357&0.9513&0.9566&0.9566&\textbf{0.9630}&0.9609&0.9607\\
		&&&PSNR ($\infty$)&21.072&22.262&22.749&22.747&\textbf{23.419}&23.194&23.174\\\cline{2-11}
		&\multirow{4}{*}{5}&\multirow{4}{*}{10}&RMSE (0)&0.0939&0.0779&0.0745&0.0746 &0.0698&0.0698&\textbf{0.0693}\\
		&&&SAM (0)&0.2294&0.1802&0.1731&0.1719&0.1704&0.1708&\textbf{0.1648}\\
		&&&CC (1)&0.9356&0.9560&0.9598&0.9597&0.9649&0.9650&\textbf{0.9654}\\
		&&&PSNR ($\infty$)&20.544&22.165&22.554&22.542&23.117&23.120&\textbf{23.184}\\\hline
		
	\end{tabular}
	\end{small}
\end{table*}

\subsubsection{Foreman}
Finally, we analyze the algorithms on the Foreman video dataset in Figure \ref{Foreman} and Table \ref{Foreman_table}. As \(R\) increases, the performance gap between MVNTF and Midas-LL1 widens, with three-step acceleration methods (3-Midas-SGD, 3-Midas-SAGA, 3-Midas-SARAH) consistently demonstrating superiority. Particularly, variance-reduced estimators  SAGA and SARAH achieve faster convergence and lower loss across all configurations, outperforming MVNTF and one-step methods, especially under {larger} \(R\) and \(L\) settings. This trend indicates that Midas-LL1 is better suited for higher numbers of components \(R\) in rank-\((L_r, L_r, 1)\) block-term tensor decomposition. While one-step methods (1-Midas-SGD, 1-Midas-SAGA, 1-Midas-SARAH) perform moderately well, they exhibit higher loss compared to three-step methods. 

\begin{figure}[!htb]
	\setlength\tabcolsep{2pt}
	\centering
	\begin{tabular}{ccc}
		\includegraphics[width=0.32\textwidth]{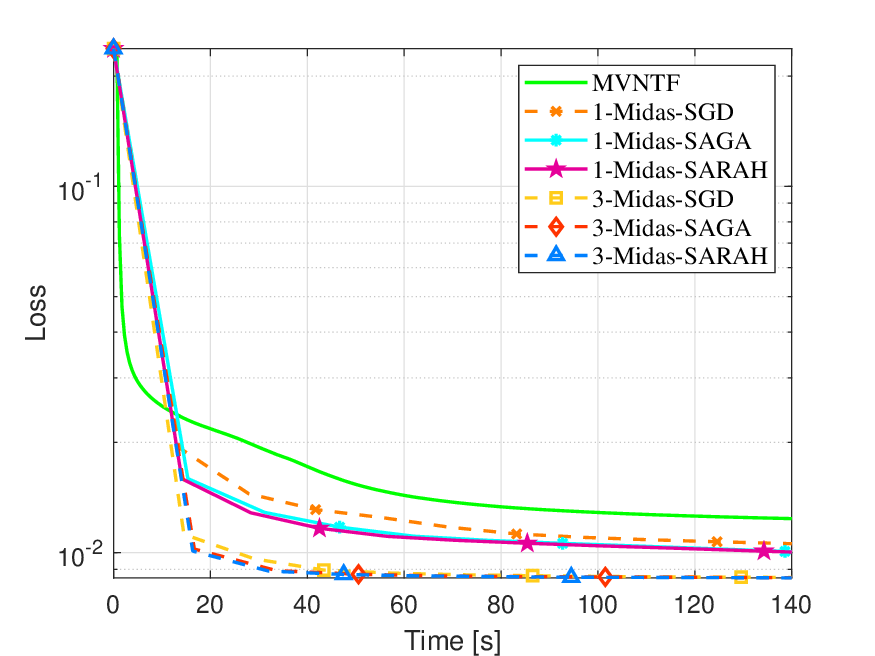}&
		\includegraphics[width=0.32\textwidth]{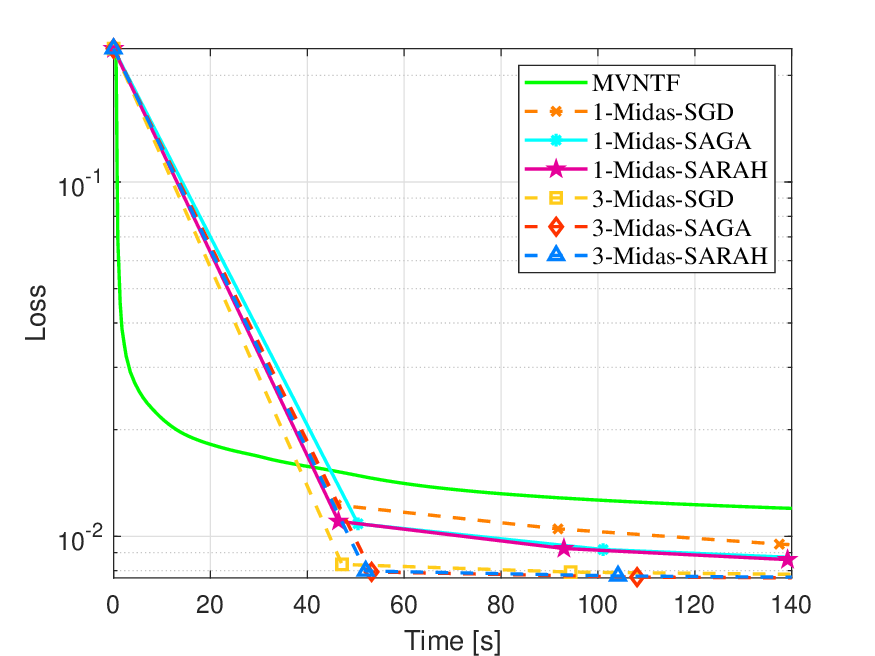}&
        \includegraphics[width=0.32\textwidth]{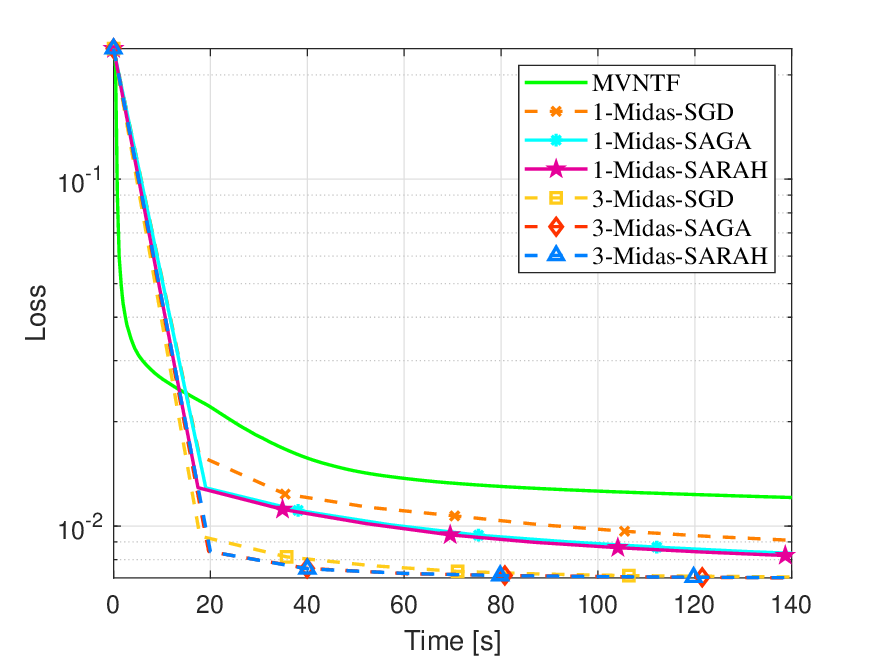}\\
		(a) $R=3$, $L=20$&(b) $R=4$, $L=15$&(c) $R=5$, $L=10$\\
	\end{tabular}
	\caption{Numerical experiments of Midas-LL1  under 1-step and 3-step acceleration with MVNTF on the $288 \times 352 \times 300 $ Foreman dataset. }
	\label{Foreman}
\end{figure}

\begin{table*}[!htb]	
	\centering
	\fontsize{7}{9}\selectfont
	\caption{
		Comparisons of Midas-LL1 under one-step and three-step inertial acceleration with MVNTF on Foreman dataset. The parameters $\alpha^k=0.3\times\frac{k-1}{k+2}$, $\beta^{k}=0.8\times\frac{k-1}{k+2}$, and $\eta^{k}=0.1$.
		``a'' denotes MVNTF;``$b_0$'' denotes 1-Midas-SGD; ``$b_1$'' denotes 1-Midas-SAGA; ``$b_2$'' denotes 1-Midas-SARAH;  ``$c_0$'' denotes 3-Midas-SGD; ``$c_1$'' denotes 3-Midas-SAGA; ``$c_2$'' denotes 3-Midas-SARAH.}
	\label{Foreman_table}
	\begin{small}
	  
	\begin{tabular}{c|c|c|c| c |c c c |c c c}
		\hline
		Dataset&$R$&$L$&Index & a&
        $b_0$ &  $b_1$ & $b_2$ &
         $c_0$ &  $c_1$ & $c_2$ 
		\cr\hline	
		&\multirow{4}{*}{3}&\multirow{4}{*}{20}&RMSE (0)&0.1530&0.1381&0.1318&0.1317 &0.1305&0.1305&\textbf{0.1304}\\
		&&&SAM (0)&0.2403&0.2167&0.2069&0.2067 &0.2046&\textbf{0.2045}&0.2047\\

		&&&CC (1)&0.7914&0.8306&0.8468&0.8469&0.8500&0.8501&\textbf{0.8503}\\
		&&&PSNR ($\infty$)&16.306&17.197&17.601&17.605&17.690&17.690&\textbf{17.691}\\ \cline{2-11}
		&\multirow{4}{*}{4}&\multirow{4}{*}{15}&RMSE (0)&0.1465&0.1292&0.1265&0.1265&0.1233&\textbf{0.1228}&0.1228\\
		&&&SAM (0)&0.2273&0.2012&0.1977&0.1977&0.1946&\textbf{0.1936}&0.1938\\
		\emph{Foreman}&&&CC (1)&0.8051&0.8527&0.8597&0.8597&0.8675&0.8688&\textbf{0.8689}\\
		&&&PSNR ($\infty$)&16.686&17.776&17.961&17.959&18.181&18.219&\textbf{18.219}\\\cline{2-11}
		&\multirow{4}{*}{5}&\multirow{4}{*}{10}&RMSE (0)&0.1457&0.1246&0.1221&0.1197 &\textbf{0.1186}&0.1186&0.1187\\
		&&&SAM (0)&0.2261&0.1958&0.1916&0.1885&0.1871&\textbf{0.1870}&0.1873\\
		&&&CC (1)&0.8075&0.8641&0.8696&0.8755&0.8782&\textbf{0.8783}&0.8780\\
		&&&PSNR ($\infty$)&16.730&18.090&18.267&18.439&18.518&\textbf{18.518}&18.508\\\hline
		
	\end{tabular}
	\end{small}
\end{table*}

\section{Conclusion}\label{conclusion}

In this paper, we proposed Midas-LL1, a multi-step inertial accelerated block-randomized stochastic gradient descent method designed for the rank-\((L_r, L_r, 1)\) block-term tensor decomposition problem. Leveraging an extended multi-step and multi-block variance-reduced stochastic estimator, we demonstrated that the proposed algorithm achieves a sublinear convergence rate for the generated subsequence. Furthermore, we established the global convergence of the sequence generated by Midas-LL1 using a novel multi-step Lyapunov function, proving that the algorithm requires at most \(\mathcal{O}(\varepsilon^{-2})\) iterations in expectation to reach an \(\varepsilon\)-stationary point. Extensive experiments on synthetic and real-world datasets validated the effectiveness of Midas-LL1, demonstrating its superior convergence behavior with multi-step inertial acceleration compared to one-step methods and existing algorithms. The results highlight the advantages of the proposed multi-step approach in achieving better performance and faster convergence.


\section*{Declarations}
{\bf Funding:} {This research is supported by  the National Natural Science Foundation of China (NSFC) grant  12401415, 12471282, 12171021, the R\&D project of Pazhou Lab (Huangpu) (Grant no. 2023K0603), and the Fundamental Research Funds for the Central Universities (Grant No. YWF-22-T-204)}.

\noindent{\bf Competing interests:} The authors have no competing interests to declare that are relevant to the content of this article.

\noindent{\bf Data  Availability Statement:} Data will be made available on reasonable request.

\bibliographystyle{abbrv}
\bibliography{main}
\end{document}